\documentclass[10pt]{article}
\usepackage{amssymb,amsmath,amsthm,epsfig}
\usepackage{latexsym, enumerate}
\usepackage{eepic}
\usepackage{epic}
\usepackage{graphicx}
\usepackage{color}
\usepackage{ifpdf}
\usepackage{subfigure}
\usepackage{tikz}
\usepackage{dsfont}
\usepackage{multirow}
\usepackage{makecell}
\usepackage{algorithm}
\usepackage{bm}
\usepackage{multirow}
\usepackage{color}
\usepackage[colorlinks, linkcolor=blue,anchorcolor=blue,citecolor=blue,urlcolor=blue]{hyperref}
\usepackage{appendix}
\usepackage{comment}
\graphicspath{{figs/}}

\theoremstyle{plain}
\newtheorem{lem}{Lemma}[section]
\newtheorem{thm}[lem]{Theorem}
\newtheorem{cor}[lem]{Corollary}

\theoremstyle{definition}
\newtheorem{defn}{Definition}[section]

\theoremstyle{remark}
\newtheorem{rem}{Remark}[section]

\newcommand{\p}{\partial}
\newcommand{\ds}{\displaystyle}
\newcommand{\vp}{\varepsilon}

\newcommand{\ms}{\medskip}
\newcommand{\R}{ \mathbb{R}}

\def \e{\ensuremath{\mathrm{e}}}
\def \i{\ensuremath{\mathrm{i}}}
\def \d{\ensuremath{\mathrm{d}}}
\begin{document}

%\title{ \large\bf  Multiphysics Cloaking : Controlling Microscale Electric and  Hydrodynamic Fields Simultaneously  in Electro-osmosis}
\title{ \large\bf  Simultaneously Cloaking Electric and Hydrodynamic Fields via Electro-osmosis}
\author{
Hongyu Liu\thanks{Department of Mathematics, City University of Hong Kong, Kowloon, Hong Kong, China.\ \ Email: hongyu.liuip@gmail.com; hongyliu@cityu.edu.hk}
\and
Zhi-Qiang Miao\thanks{School of Mathematics, Hunan University, Changsha 410082, Hunan Province, China. Email: zhiqiang\_miao@hnu.edu.cn}
\and
Guang-Hui Zheng\thanks{School of Mathematics, Hunan University, Changsha 410082, Hunan Province, China. Email: zhenggh2012@hnu.edu.cn; zhgh1980@163.com}
}
\date{}%
\maketitle
% ----------------------------------------------------------------
\begin{abstract}
In this paper, we develop a general mathematical framework for the electro-osmosis problem to design simultaneous microscale electric and hydrodynamic cloaking in a Hele-Shaw configuration. A novel approach to achieving simultaneously cloaking both the electric and flow fields through a combination of scattering-cancellation technology and an electro-osmosis effect is proposed. In the design, the electric field is manipulated with scattering-cancellation technology while the pressure with electro-osmosis effect. As proof of this concept, the perfect electric and hydrodynamic cloaking conditions are derived for the cloaks with the cross-sectional shape being annulus or confocal ellipses using the layer potential techniques. Furthermore, we also propose an optimization scheme for the design of approximate cloaks within general geometries and prove the well-posedness of the optimization problem. In particular, the conditions that can ensure the simultaneous occurrence of approximate cloaks for general geometries are also established.  Our theoretical findings are validated by a variety of numerical results and guide efficiently designing electric-related multiphysics cloaking.
\end{abstract}

\smallskip
{\bf keywords}: Multiphysics cloaking; electro-osmosis; layer potential; optimization scheme.
% ----------------------------------------------------------------
\section{Introduction}
Invisibility cloaking has long been a topic of intense study, which draws human curiosity. There are many studies in the fields of physics and mathematics
in various scientific areas, such as acoustic waves \cite{ammari2013, deng2017, Kohn2010, Liu2009}, conductive heat flux \cite{Craster2017, Narayana2012}, electromagnetic waves\cite{bao2014, deng2017(1), greenleaf2009, Leonhardt2006, Pendry2006}, pressure field in an elastic medium \cite{li2016, li2018, Stenger2012}, dc electric currents \cite{Yang2012}, and matter waves \cite{Zhang2008}.
 Most of the early research on invisibility cloaking struggled to satisfy the requirements of one set of physical equations that only controlled and realized one physical phenomenon. However, in the real world, multiphysics is present in many aspects of daily life, industry, and nature,  and it is crucial to take this into account when developing practical applications. Hence, creating a multiphysics cloak that can simultaneously make objects invisible from  multiple physical field measurements needs to be solved, which is a fascinating and challenging subject.
Based on the same Laplacian forms of governing equations in both electrical conduction and the thermal diffusion processes, Li et al. \cite{li2010} first explored whether the same structure can realize a cloaking effect in these two fields simultaneously. Then, Moccia et al. \cite{Ma2014} experimentally constructed a multi-physics cloak that breaks the limitation that the cloaking devices can only cloak a single physical field.
Meanwhile, research on multi-physical field control is gradually expanding from thermal electrostatic fields \cite{li2010, Ma2014} to other physical fields, such as carpet cloak for electromagnetic, acoustic and water waves simultaneously \cite{Xu2015, Yang2016}, thermo-hydrodynamic cloak \cite{Wang2021, Yeung2020}, magnetostatic-acoustic cloak \cite{Zhou2020a}, and electromagnetic-acoustic cloak \cite{Fujii2022, Sun2020, Zhou2020b}.
In addition, several experiments to cloak an object in multi-physical flows have been attempted using a combination of passive and active schemes \cite{Lan2016} and optimization methods \cite{Fujii2019, Xu2020} in order to overcome the complicated structural challenges of metamaterials.
Although all these works try to fabricate cloaking devices that can be applicable in multiple physical fields, it still lacks a general theory, such as, transformation optics and scattering-cancellation technology, to achieve multi-physical field cloaking.
Therefore, there are a lot of works to be studied in multiphysics cloaking.

Recently, Boyko et al. \cite{Boyko2021} proposed a novel theoretical framework and an experimental demonstration of hydrodynamic cloaking and shielding in a Hele-Shaw cell via electro-osmosis. The method has attracted our attention. We first develop a general mathematical framework \cite{Liu2023a} for perfect and approximate hydrodynamic cloaking and shielding of electro-osmotic flow in the spirit of Boyko's work. Then, we address the concept of enhanced near-cloaking \cite{Liu2023b} in the context of microscale hydrodynamics using electro-osmosis by the perturbation theory. Note that electro-osmosis is a coupled multiphysics process involving electric, pressure and flow fields.
However, in \cite{Boyko2021, Liu2023a, Liu2023b} authors only focus on the hydrodynamic cloaking, but ignore the electric field cloaking. From the inverse problem perspective, it is true that one can not detect an object surrounded by a hydrodynamic cloak by pressure and flow field measurements, but this object can be detected by electric field measurement. Therefore, simultaneous electric cloaking is very necessary. Based on this background and motivation, in this paper we study the simultaneous electric and hydrodynamic cloaking in electro-osmosis model.
More importantly, in the electro-osmosis model, the pressure and flow fields are dependent on the electric field, but not vice versa.
This so-called one-way coupling between the electric field and flow field allows us to analyze electric cloaking independent of hydrodynamic cloaking, if desired. Hence, before proceeding to
analyze hydrodynamic cloaking, we first present the theory of pure electric cloaking.
Throughout this paper, we apply the scattering-cancellation technology \cite{Alù2005, Chen2012} to construct an electric cloak with a monolayer structure and homogeneous isotropic material. By solving the governing equations directly, we formulate the specific parameter requirements for desired cloaking functionality. After the occurrence of electric cloaking, we further derive the hydrodynamic cloaking conditions using the methods in \cite{Boyko2021, Liu2023a}. It is worth noting  that the hydrodynamic cloaking conditions will be changed due to the changes in the electric field compared with that of \cite{Liu2023a}.

The purpose of this paper is to establish a more general mathematical framework for simultaneous microscale electric and hydrodynamic cloaking.
Our approach extends to the method first provided in \cite{Liu2023a} to achieve multiphysics cloaking for the electro-osmosis problem from multiphysical field measurements. It is based on the multi-cloak structure \cite{Raza2015} in which each cloak works as an invisibility cloak for a specific physical phenomenon. The design blueprint of the multiphysics cloaking structure is presented briefly as follows. We first design a structure coated around a perfect insulator to achieve electric cloaking and then design the hydrodynamic cloak by the method provided in \cite{Liu2023a} when the electric cloaking happens.  Finally, the multiphysics cloaking structure is obtained by wrapping the hydrodynamic cloak around an object which consists of a perfect insulator and electric cloak.
For more details, the contributions of this work are fourfold:
\begin{itemize}
\item{We present the rigorous mathematical definition of simultaneous electric and hydrodynamic cloaking, based on the physics and mathematics literature \cite{Boyko2021, Liu2023a}.}
\item{The representation formula of the solution of the coupled system is established by single-layer potential, which can ensure the well-posedness of the electro-osmosis problem  and gives a quantitative description.}
\item{In the so-called layer approach, we can determine sharp conditions that can guarantee the simultaneous occurrence of both electric and hydrodynamic cloaking for annulus (radial case) and confocal ellipses (non-radial case). In particular, for the confocal ellipses case which is not covered in \cite{Boyko2021}, we introduce an additional elliptic coordinates technique to address the challenge caused by non-radial geometry.}
\item{Furthermore, we provide an optimization technique for designing simultaneous electric and hydrodynamic cloaking and demonstrate the well-posedness of the corresponding optimization problem for more general geometry.  More importantly, we also establish the condition under which approximation cloaks are guaranteed to occur simultaneously for arbitrary geometry. Numerous numerical experiments including smooth objects, non-smooth objects and multiple objects, indicate that the optimized zeta potential can achieve approximate electric and hydrodynamic cloaking simultaneously.}
\end{itemize}

The rest of the paper is organized as follows. We begin with the setting of the problem and main results in Section \ref{sec:problem}. In Section \ref{sec:layer-potentials}, we first review some preliminary knowledge on boundary layer potentials and then establish the representation formula of the solution of the governing equations.
In Section \ref{sec:simultaneous-cloaking}, we devote to the study of the simultaneous electric and hydrodynamic cloaking conditions by the analytical solutions and the optimal method, respectively.
In Section \ref{sec:NumSim}, we provide numerical examples to justify our theoretical results. The paper ends with a conclusion.

\section{Setting of the problem and main results}\label{sec:problem}
We begin with conventional electro-osmosis and introduce some notations.
Consider two parallel insulating plates of length $\tilde{L}$ and width $\tilde{W}$ with a thin gap, defined as the $\tilde{x}_3 = 0, \tilde{h}$ planes in a system of Cartesian coordinates $\tilde{x} =(\tilde{x}_1, \tilde{x}_2, \tilde{x}_3) \in \R^3$, confining an electrolyte solution of density $\tilde{\rho}$, dielectric permittivity $\tilde{\varepsilon}_m$, and viscosity $\tilde{\mu}$.  The structure forms a Hele-Shaw configuration in $\R^3$, as shown in Figure \ref{fig-schematic}(a). The fluid is assumed incompressible and a low Reynolds number description is valid.
The arbitrary zeta-potential distribution in the lower and upper plates are given by $\tilde{\zeta}^L = \tilde{\zeta}^L(\tilde{x})$ and $\tilde{\zeta}^U = \tilde{\zeta}^U(\tilde{x})$. Let $\tilde{\bm{v}} = (\tilde{\bm{u}}(\tilde{x}, \tilde{t}),\tilde{w}(\tilde{x}, \tilde{t}))$ and $\tilde{p} = \tilde{p}(\tilde{x})$ denote the velocity field and  the pressure of the fluid,  respectively.  Here $\tilde{\bm{u}}(\tilde{x}, \tilde{t})$ is  the in-plane velocity field. If an electrostatic in-plane electric field $\tilde{\bm E}$ is applied
parallel to the plates, the fluid motion is then governed by the continuity and momentum equations
\begin{equation*}
  \tilde{\nabla}\cdot \tilde{\bm v} = 0, \qquad \tilde{\rho}\Big( \frac{\p \tilde{\bm v}}{\p \tilde{t}}+ \tilde{\bm v}\cdot  \tilde{\nabla}\tilde{\bm v}\Big)
 =  -  \tilde{\nabla} \tilde{p} + \tilde{\mu} \tilde{\Delta}  \tilde{\bm v},
\end{equation*}
with  the Helmholtz--Smoluchowski slip boundary conditions
\begin{equation*}
  \tilde{\bm{u}}|_{\tilde{x}_3=0}= -\frac{\tilde{\varepsilon}_m\tilde{\zeta}^L \tilde{\bm E}}{\tilde{\mu}}, \qquad   \tilde{\bm{u}}|_{\tilde{x}_3=\tilde{h}}= -\frac{\tilde{\varepsilon}_m\tilde{\zeta}^U \tilde{\bm E}}{\tilde{\mu}},
\end{equation*}
where we shall  deal with electromagnetic phenomena in the electrostatic regime because the fluid in the bulk is electrically neutral and the electric displacement $\tilde{\bm D}$ is solenoidal, i.e., $\tilde{\nabla} \times \tilde{\bm E}  = 0$ and $\tilde{\nabla} \cdot \tilde{\bm D}  = 0$.  The electric field can be expressed through an electrostatic potential $\tilde{\varphi}$,  $\tilde{\bm E} = -\tilde{\nabla} \tilde{\varphi} $,  that is governed by the Laplace equation $\tilde{\nabla}\cdot (\tilde{\vp}\nabla\tilde{\varphi}) =0$ using the constitutive relation $\tilde{\bm D}= \tilde{\vp}\tilde{\bm E}$, which satisfies the insulation boundary conditions on the walls, $\frac{\p \tilde{\varphi} }{\p \tilde{x}_3}|_{\tilde{x}_3=0, \tilde{h}} = 0$.
%By making use of the Helmholtz--Smoluchowski slip condition, we have implicitly assumed that surface conduction
%is negligible and thus consider asymptotically small Dukhin numbers.
 For convenience in what follows, throughout this paper, we mark the dimensional variables that appear in these equations with a tilde, for instance, $\tilde{\bm u}$, $\tilde{p}$, and so on. Moreover, we also denote the dimensional gradient and Laplace operator as $\tilde{\nabla}$ and $\tilde{\Delta}$.  In contrast to this, we mark these dimensionless variables and operators  without a tilde.

We next derive dimensionless forms of the governing creeping-flow equations and boundary conditions for the  electro-osmosis problem.
In microscale flows, fluidic inertia is commonly negligible compared to viscous stresses. Further, the gap is narrow. Therefore, we restrict our analysis to shallow geometries ($\tilde{h}\ll \tilde{L}, \tilde{W}$) and neglect fluidic inertia, indicated by a small reduced Reynolds number.
By utilizing the lubrication approximation theory, we average over the depth of the cell and reduce the analysis to a two-dimensional problem in $\R^2$. The governing equations for the depth-averaged velocity $\bm \tilde{\bm u}_{aver}$, the pressure $\tilde{p}$, and the electrostatic potential $\tilde{\varphi}$ in $\tilde{x}_1$--$\tilde{x}_2$ plane are (see \cite{Boyko2021} )
\begin{equation}\label{dimension-equations}
 \tilde{\bm u}_{aver} = - \frac{\tilde{h}^2}{12 \tilde{\mu}} \tilde{\nabla} \tilde{p} + \tilde{\bm u}_{slip},  \quad \tilde{\Delta}\tilde{p} = \frac{12 \tilde{ \varepsilon}}{\tilde{h}^2}\tilde{\nabla}\tilde{\varphi}\cdot \tilde{\nabla}\tilde{\zeta}_{mean} \quad \mbox{and} \quad \tilde{\nabla}\cdot (\tilde{\vp}\nabla\tilde{\varphi}) = 0,
\end{equation}
where $\bm \tilde{\bm u}_{aver}$ is the average value of $\tilde{\bm u}$ over $\tilde{x}_3$ by integration, $\tilde{\bm u}_{slip}=\frac{\tilde{\varepsilon}}{\tilde{\mu}}\tilde{\zeta}_{mean}\tilde{\nabla}\tilde{\varphi}$ is the depth-averaged Helmholtz-Smoluchowski slip velocity, $\tilde{\zeta}_{mean}$ is the arithmetic mean value of the zeta potential on the lower and upper plates, i.e., $\tilde{\zeta}_{mean}=(\tilde{\zeta}^L+\tilde{\zeta}^U)/2$. In the subsequent analysis, we assume that the zeta potential on the lower and upper plates is the same.
%- \zeta \nabla \varphi,
Scaling by the characteristic dimensions, we introduce the following dimensionless equations:
\begin{equation}\label{dimensionless-equations}
 \bm u_{aver}  =- \frac{1}{12} \nabla p - \zeta_{mean}  \nabla \varphi,\quad \Delta p = -12 \nabla \varphi \cdot \nabla \zeta_{mean}  \quad \mbox{and} \quad \nabla \cdot(\vp \nabla\varphi) = 0,
\end{equation}
where $\bm u_{aver}$, $p$, $\varphi$ and $\zeta_{mean}$ are non-dimensional normalized variables.

Besides these equations and the geometry of the flow domain, we also need  the boundary conditions that apply at the boundaries of the domain to fully describe the flow. These boundary conditions contain the form and magnitudes of the velocity on the left inlet and right outlet to the domain.
In this paper, we mainly consider a pillar-shaped object with an arbitrary cross-sectional shape confined between the walls of a Hele-Shaw cell and subjected to a pressure-driven flow with an externally imposed mean velocity $\tilde{u}_{ext}$ and electric field $\tilde{E}$ along the $\tilde{x}_1$--axis, as illustrated in Figure \ref{fig-schematic}. The reduced two-dimensional problem is shown in Figure \ref{fig-schematic}(b). We solve the problem assuming an unbounded domain, enabled by the fact that the boundaries of the chamber are located far from the cloaking region.
\begin{figure}[H]
	\centering  %图片全局居中
	%\subfigbottomskip=2pt %两行子图之间的行间距
	\subfigcapskip=-10pt %设置子图与子标题之间的距离
	\subfigure[]{
		\includegraphics[width=0.5\linewidth]{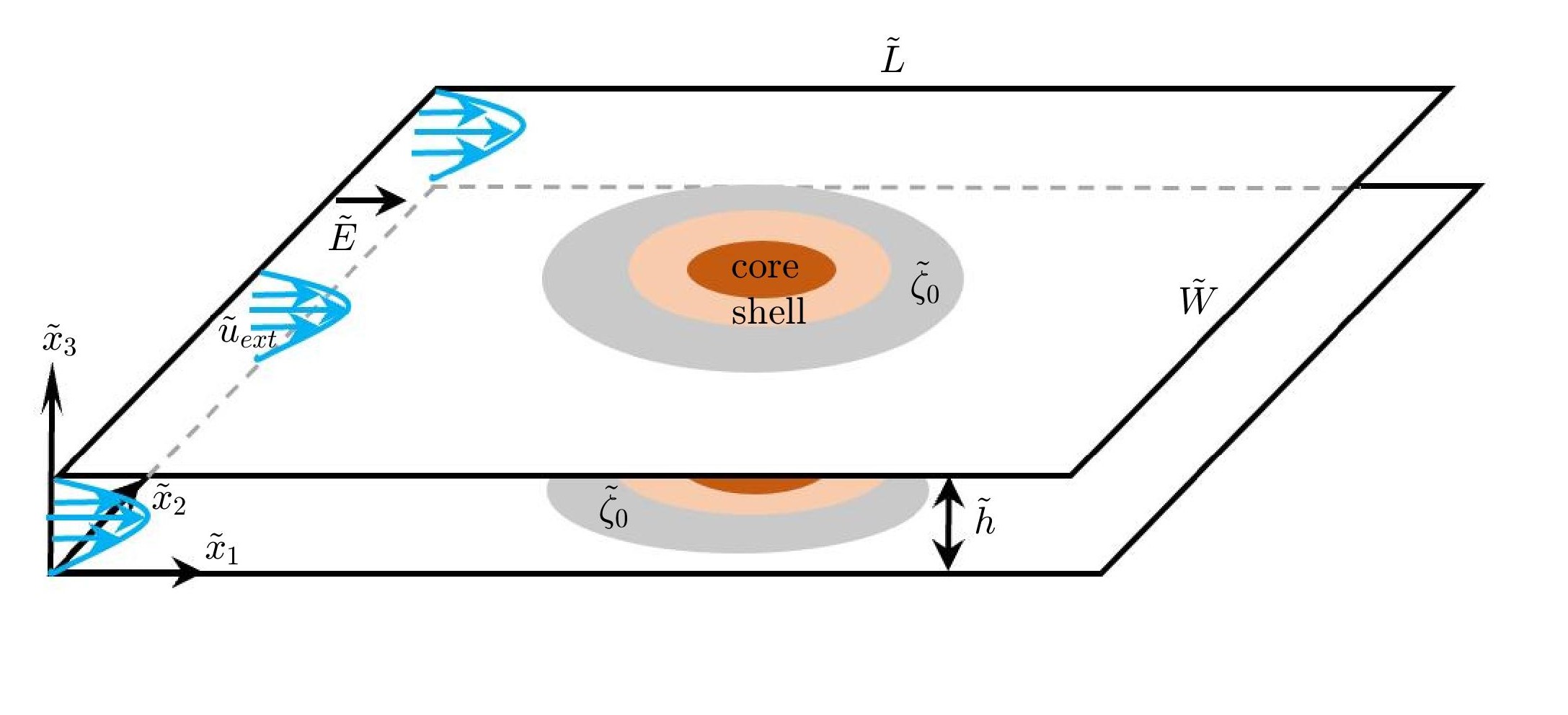}}
	%\quad
	\subfigure[]{
		\includegraphics[width=0.35\linewidth]{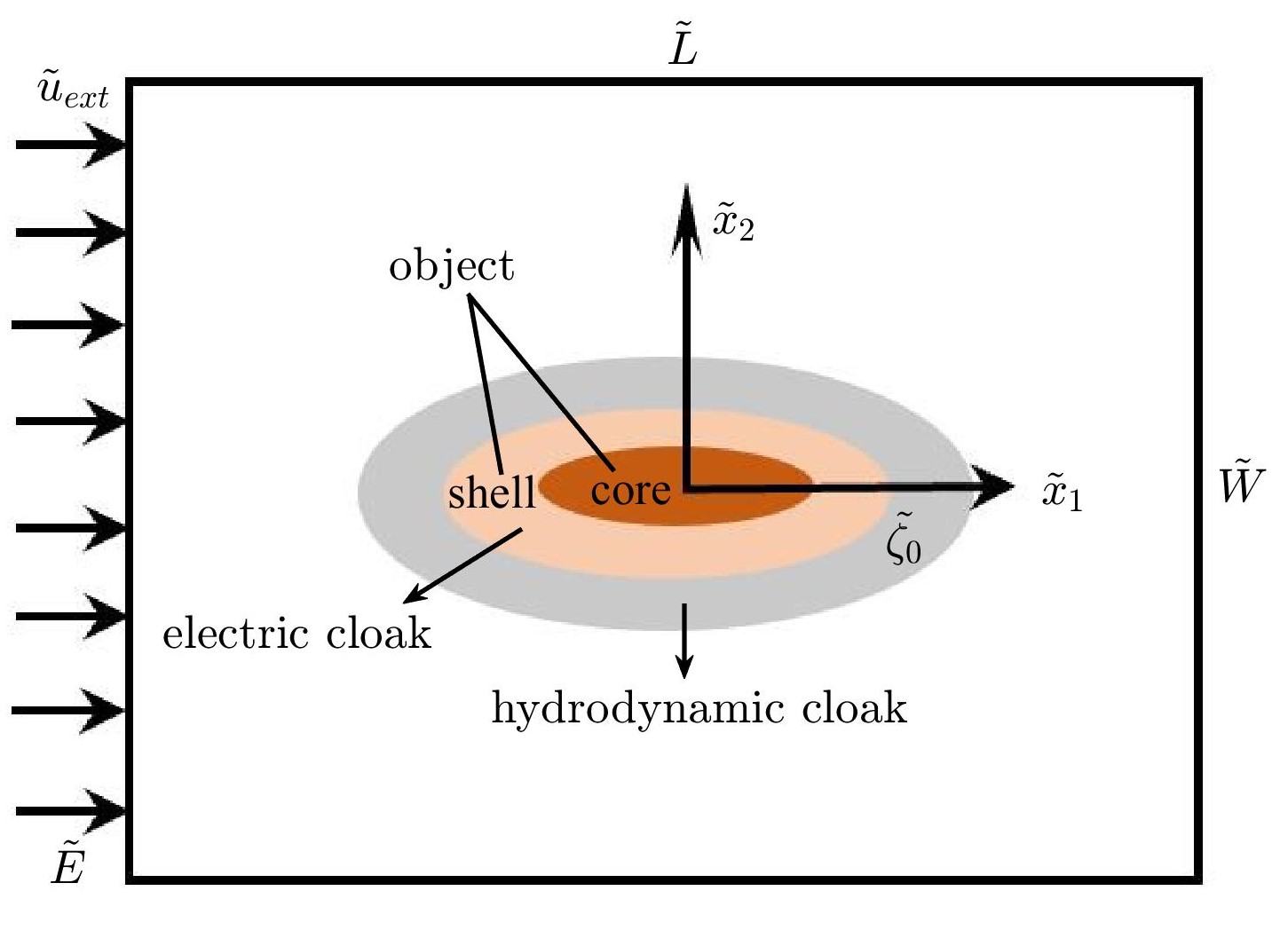}}
	\caption{Schematic illustration of the Hele-Shaw configuration. (a) the three-dimensional electro-osmosis model. (b) the reduced two-dimensional problem.
 }\label{fig-schematic}
\end{figure}

%(or the form and magnitude of the velocity and electric field far from object in an unbounded domain).
%In this paper, we mainly consider a pillar-shaped object with arbitrary cross-sectional shape confined between the walls of a Hele-Shaw cell and subjected to a pressure-driven flow with an externally imposed mean velocity $\tilde{u}_{ext}$ and electric field $\tilde{E}$ along the $\tilde{x}_1$--axis, as shown in Figure \ref{fig-schematic}. The reduced two-dimensional problem is shown in Figure \ref{fig-schematic}(b).
In order to cloak the object electrically and hydrodynamically, we are concerned with the scattering by the insulation core coated by a shell and the object surrounded by a region, i.e., with exterior boundary value problems for the Laplace equation.
To mathematically state the problem, let  $B$ (core),  $D$ (object)  and $\Omega$ be  bounded domains in $ \mathbb{R}^2$ with $\overline{B}\subseteq\overline{D}\subseteq \Omega$. Throughout this paper, we assume that  $B$, $D$ and $\Omega$ are of class $C^{1,\alpha}$ for some $0<\alpha<1$.  Let $H(x)$ and $P(x)$ be the harmonic functions  in $\mathbb{R}^2$, denoting the background electrostatic potential and pressure field.
The permittivity distribution in $ \mathbb{R}^d\setminus\overline{D}$ is given by
\begin{equation*}
\varepsilon =
  \begin{cases}
\ds \varepsilon_m &\quad \mbox{in } \mathbb{R}^d\setminus\overline{D},\\
\ds\varepsilon_s &\quad \mbox{in } D \setminus\overline{B},\\
\ds 0 &\quad \mbox{in } B.
  \end{cases}
\end{equation*}
For a given constant parameter $\zeta_0 \in \R$, the  zeta potential distribution in
$\mathbb{R}^2\setminus\overline{D}$ is given by
\begin{align*}
          \zeta_{mean}   =
            \begin{cases}
            \ds \zeta_0 & \quad \mbox{in } \Omega \setminus\overline{D},\ms\\
            \ds 0 & \quad \mbox{in } \mathbb{R}^2\setminus\overline{\Omega}.
            \end{cases}
\end{align*}

We may consider the configuration as an insulation core $B$ coated by the shell $D \setminus\overline{B}$ with permittivity $\varepsilon_s$ and a no-penetration object $D$ coated by the control region $\Omega \setminus\overline{D}$  with zeta potential $ \zeta_0$.  From the equation \eqref{dimensionless-equations} and the assumption of unbounded domain, the governing equations for non-uniform electro-osmotic flow via a Hele-Shaw configuration is modeled as follows:
\begin{align}\label{electro-osmotic equation}
\begin{cases}
\ds \nabla \cdot(\vp \nabla\varphi)= 0 \quad \ &\mbox{in} \  \mathbb{R}^2\setminus\overline{B},\ms\\
\ds \frac{\partial \varphi}{\partial \nu} = 0  & \mbox{on } \partial B, \ms \\
\ds \varphi|_{+}=\varphi|_{-} & \mbox{on }  \partial D,\ms\\
\ds \varepsilon_m\frac{\partial \varphi}{\partial \nu} \Big|_{+} = \varepsilon_s \frac{\partial \varphi}{\partial \nu} \Big|_{-}  & \mbox{on } \partial D,\ms \\
\ds  \varphi = H(x) + O(|x|^{-1}) & \mbox{as } |x|\rightarrow + \infty,\ms\\
%\end{cases}
%\end{align}
%and
%\begin{align}\label{press}
%\begin{cases}
\ds \Delta p= 0 & \mbox{in }  \mathbb{R}^2\setminus\overline{D}, \ms\\
\ds \frac{\partial p}{\partial \nu} = -12 \zeta_0\frac{\partial \varphi}{\partial \nu}  & \mbox{on } \partial D, \ms\\
\ds p|_{+}=p|_{-} & \mbox{on }  \partial \Omega,\ms\\
\ds \frac{\partial p}{\partial \nu} \Big|_{+} - \frac{\partial p}{\partial \nu} \Big|_{-} = 12 \zeta_0\frac{\partial \varphi}{\partial \nu} & \mbox{on } \partial \Omega,\ms \\
\ds  p = P(x) + O(|x|^{-1}) & \mbox{as } |x|\rightarrow +\infty,
\end{cases}
\end{align}
where $ \frac{\partial }{\partial \nu}$ denotes the outward normal derivative and we use the notation $\frac{\partial p}{\partial \nu}\big|_{\pm}$ indicating
$$
\frac{\partial p}{\partial \nu}\bigg|_{\pm}(x):=\lim_{t\rightarrow 0^+}\langle \nabla p(x\pm t\nu(x)),\nu(x) \rangle, \ \ x\in \p \Omega,
$$
where $\nu$ is the outward unit normal vector to $\p \Omega$.

We are now in a position to introduce the definition of electric and hydrodynamic cloaking which plays a central role in this paper.
\begin{defn}\label{def-cloaking}
The quintuples $\{B, D, \Omega; \varepsilon_s,\zeta_0\}$ is said to be a perfect electric and hydrodynamic cloaking, if
\begin{equation}\label{cond-cloaking}
\bm E = -\nabla H \quad \mbox{in }  \mathbb{R}^2\setminus\overline{D}\quad \mbox{and}\quad
\bm u_{aver}  = - \nabla P / 12 \quad  \mbox{in }  \mathbb{R}^2\setminus\overline{\Omega},
\end{equation}
where $\bm u_{aver}$ is the dimensionless version of  $\tilde{\bm u}_{aver}$ defined in \eqref{dimension-equations}.
If the notation $"="$ is replaced by $"\approx"$, then it is called a near/approximate hydrodynamic cloaking.
\end{defn}

Away from the hydrodynamic cloaking region, the pressure is related to the velocity field through $ \bm{u}_{aver} = -\nabla p/12 $ subjected to the
boundary condition $p(x)=P(x)$ as $|x|\rightarrow\infty$, and therefore, according to the Definition \ref{def-cloaking}, the condition (\ref{cond-cloaking}) can be expressed in terms of the electrostatic potential and pressure as
\begin{align}\label{cond-cloaking-p}
\varphi = H \quad \mbox{in }  \mathbb{R}^2\setminus\overline{D}\quad \mbox{and}\quad p(x)=P(x), \quad x\in \mathbb{R}^2\setminus\overline{\Omega}.
\end{align}
In this paper, we assume $B$, $D$ and $\Omega$ are known, and want to find appropriate permittivity $\varepsilon_s$ and zeta potential $\zeta_0$ to achieve the electric and hydrodynamic cloaking simultaneously.

Our main results in this paper are given in the following theorems. The proofs are given in Section \ref{subsec:cloaking-annulus} \ref{subsec:cloaking-ellipses} and \ref{subsec:optimal}, respectively.
\begin{thm}\label{main-thm-annulus}
Let the domains $B$, $D$ and $\Omega$ be concentric disks of radii $r_o$, $r_i$ and $r_e$, where $r_e>r_i>r_o$. Let $H(x) = r^n\e^{\i n \theta}$ and $P(x) = 12 r^n\e^{\i n \theta}$ for $n\geq 1$. If
\begin{equation}\label{annulus-cloaking-zeta}
\varepsilon_s=\frac{r_i^{2n}+r_o^{2n}}{r_i^{2n}-r_o^{2n}}\varepsilon_m \quad \mbox{and} \quad \zeta_0=\frac{2r_i^{2n}}{r_e^{2n}-r_i^{2n}},
\end{equation}
then the perfect electric and hydrodynamic cloaking occur simultaneously.
\end{thm}
%\begin{rem}
%In \cite{Boyko2021}, the authors consider the special case where the background electrostatic potential and pressure field are given by $H(x) = r \cos(\theta)$ and $P(x) = 12r\cos(\theta)$. In fact, the special case is included in Theorem \ref{main-thm-annulus} when $n=1$. Compared with the linear background fields in \cite{Boyko2021}, we extend the background electrostatic potential and pressure field to a more general harmonic function in Theorem \ref{main-thm-annulus}.
%\end{rem}
%Furthermore, we first consider the case of the confocal ellipses under the general background electrostatic potential and pressure field. The results are summarized in the following theorem.

\begin{thm}\label{main-thm-ellipses}
Let the domains $B$, $D$ and $\Omega$ be confocal ellipses of elliptic radii $\xi_o$, $\xi_i$ and $\xi_e$, where $\xi_e>\xi_i>\xi_o$.
\begin{itemize}
    \item Let $H(x) = \cosh (n\xi) \cos (n\eta)$ and $P(x) = 12  \cosh (n\xi) \cos (n\eta)$ for $n\geq 1$. If
    \begin{equation}\label{ellipse-cloaking-zeta-x}
   \varepsilon_s = \tanh(n\xi_i)\coth(n(\xi_i-\xi_o))\varepsilon_m\quad\mbox{and}\quad \zeta_0=\frac{ \sinh (n\xi_i)}{\sinh (n\xi_e)\cosh (n(\xi_e -\xi_i))-\sinh (n\xi_i)},
    \end{equation}
   then the perfect electric and hydrodynamic cloaking occur simultaneously.
    \item Let $H(x) = \sinh (n\xi) \sin (n\eta)$ and $P(x) = 12  \sinh (n\xi) \sin (n\eta)$ for $n\geq 1$. If
    \begin{equation}\label{ellipse-cloaking-zeta-y}
    \varepsilon_s = \coth(n\xi_i)\coth(n(\xi_i-\xi_o))\varepsilon_m\quad\mbox{and}\quad \zeta_0=\frac{ \cosh (n\xi_i)}{\cosh (n\xi_e)\cosh (n(\xi_e -\xi_i))-\cosh (n\xi_i)},
    \end{equation}
  then the perfect electric and hydrodynamic cloaking occur simultaneously.
  \end{itemize}
\end{thm}

By using the optimization approach, we shall establish a sufficient condition for the simultaneous occurrence of approximate electric and hydrodynamic cloaking in the following theorem for the arbitrary geometry for $B$, $D$ and $\Omega$.
%We only give proof of the case of cloaking. In a similar way, one can prove the case of shielding.
\begin{thm}\label{thm-estimation}
Let $p$ be the solution to (\ref{electro-osmotic equation}) with  $\varphi|_{+}=H$ on $\p D$, $\frac{\p p}{\p \nu}|_{+}=\frac{\p P}{\p \nu}$ on $\p \Omega$, and the optimization functional $\mathcal{G}(\varepsilon_{s,opt})$, $\mathcal{F}(\zeta_{0,opt})$ are defined by
(\ref{cost-functional-ep-cloaking}), (\ref{cost-functional-cloaking}) respectively.
If there exists an optimal permittivity and zeta potential pair $(\varepsilon_{s,opt},\zeta_{0,opt})$ such that $\mathcal{G}(\zeta_{0,opt})<\epsilon^2$ and $\mathcal{F}(\zeta_{0,opt})<\epsilon^2$  where $\epsilon\ll 1$, then the approximate electric and hydrodynamic cloaking occur simultaneously, that is, $|\varphi - H|<\epsilon$ in $\R^2 \setminus \overline{D}$  and $|p - P|<\epsilon$ in $\R^2 \setminus \overline{\Omega}$.
\end{thm}
\begin{rem}
It's worth noting that achieving the perfect electric and hydrodynamic cloaking simultaneously is typically challenging for general geometry. It is more strongly connected to the corresponding overdetermined boundary value problem (see Section \ref{subsec:optimal} for details).
\end{rem}

\section{Layer potentials formulation}\label{sec:layer-potentials}
In this section, we introduce some preliminary knowledge on boundary layer potentials, which shall be needed in establishing the representation formula of the solution of the governing equations.
The fundamental solution $G(x,y)$ to the Laplace operator in  $\mathbb{R}^2$ is given by
\begin{align*}
G(x,y)=\frac{1}{2\pi}\ln|x-y|.
\end{align*}
Let $\Gamma_o := \p B$, $\Gamma_i := \p D$ and $\Gamma_e := \p \Omega$. For $\Gamma = \Gamma_o$, $\Gamma = \Gamma_i$ or $\Gamma_e$, let us now introduce the single-layer potential by
\begin{align*}
\mathcal{S}_\Gamma[\vartheta](x) :=\int_{\Gamma}G(x,y)\vartheta(y)\d s(y), \quad  x\in \mathbb{R}^2,
\end{align*}
for some surface density $\vartheta\in L^2(\Gamma)$.
We also define a boundary integral operator $\mathcal{K}^*_{\Gamma}$ on $L^2(\Gamma)$ by
\begin{align*}
%\label{NP}
\mathcal{K}_\Gamma^*[\vartheta](x)=\int_{\Gamma}\frac{\partial G(x,y)}{\partial \nu (x)}\vartheta(y)\d s(y), \quad x \in \Gamma,
\end{align*}
which is sometimes called the Neumann-Poincar$\acute{e}$ (NP) operator on $\Gamma$.
The following jump formula relates the traces of the normal derivative of the single layer potential to the operator $\mathcal{K}^*_{\Gamma}$. We have
\begin{align*}
%\label{jump-relation}
\frac{\partial\mathcal{S}_\Gamma[\vartheta]}{\partial \nu}\bigg|_{\pm}(x)&=\Big(\pm\frac{1}{2}I+\mathcal{K}^*_\Gamma\Big)[\vartheta](x), \ \ x\in \Gamma.
\end{align*}

Recall Kellogg's result in \cite{Kellogg1953} on the spectrums of  $\mathcal{K}^*_{\Gamma}$, we know that the
eigenvalues of $\mathcal{K}^*_{\Gamma}$ lie in the interval $(-1/2, 1/2]$.  Hence, we have the following lemma, which is useful to establish the representation formula of solution to (\ref{electro-osmotic equation}).
\begin{lem}\label{I+K-invertible}\cite{Ammari2007}
The operator $\frac{1}{2} I+\mathcal{K}^*_{\Gamma_i}: L_0^2(\Gamma_i)\rightarrow L_0^2(\Gamma_i)$ is invertible. Here $L_0^2 := \{f\in L^2(\Gamma_i); \int_{\Gamma_i} f \d s=0\}$.
% $L_0^2$ is the collection of all continuous functions with the integral zero.
\end{lem}

Using Lemma \ref{I+K-invertible}, we now give the following representation theorem by the layer potential theory. The proof is similar to that of Theorem 3.2 in \cite{Liu2023a}, thus we omit its proof.
\begin{thm}\label{well-posedness}
Let $\varphi\in C^2(\mathbb{R}^2\setminus\overline{B})\bigcap C(\mathbb{R}^2\setminus B)$, $p\in C^2(\mathbb{R}^2\setminus\overline{D})\bigcap C(\mathbb{R}^2\setminus D)$ be the classical solution to (\ref{electro-osmotic equation}).
Then $\varphi$ can be represented as
\begin{equation}\label{sol-ep}
\varphi =
H(x) + \mathcal{S}_{\Gamma_o}[\phi_o](x)+ \mathcal{S}_{\Gamma_i}[\phi_i](x),\quad x\in\mathbb{R}^2\setminus\overline{B},
\end{equation}
where the density function pair $(\phi_o, \phi_i)\in L_0^2(\Gamma_o)\times L_0^2(\Gamma_i)$ satisfy
\begin{align*}
%\label{varphi-density-equation}
\begin{cases}
\ds \Big(\frac{1}{2} I+\mathcal{K}^*_{\Gamma_o}\Big)[\phi_o]
+\frac{\partial\mathcal{S}_{\Gamma_i}[\phi_i]}{\partial\nu_o}= -\frac{\p H}{\p \nu_o} &\quad  \mbox{on }\Gamma_o, \ms\\
\ds \frac{\partial\mathcal{S}_{\Gamma_o}[\phi_o]}{\partial\nu_i} + \big(\lambda I + \mathcal{K}^*_{\Gamma_i}\big)[\phi_i] = -\frac{\p H}{\p \nu_i} &\quad  \mbox{on }\Gamma_i, \
\end{cases}
\end{align*}
where $\lambda = \frac{\varepsilon_m+\varepsilon_s}{2(\varepsilon_m-\varepsilon_s)}$.

And $p$ can be represented using the single-layer potentials $S_{\Gamma_i}$ and $S_{\Gamma_e}$ as follows:
 \begin{equation}\label{sol-p}
p = P(x) +\mathcal{S}_{\Gamma_i}[\psi_i](x) + \mathcal{S}_{\Gamma_e}[\psi_e](x),\quad  x\in  \mathbb{R}^2\setminus\overline{D},
\end{equation}
where the pair $(\psi_i, \psi_e)\in L_0^2(\Gamma_i)\times L_0^2(\Gamma_e)$ satisfy
\begin{align}
\label{density-equation}
\begin{cases}
\ds \Big(\frac{1}{2} I+\mathcal{K}^*_{\Gamma_i}\Big)[\psi_i]
+\frac{\partial\mathcal{S}_{\Gamma_e}[\psi_e]}{\partial\nu_i}= -\frac{\p P}{\p \nu_i} - 12 \zeta_0\frac{\partial \varphi}{\partial \nu_i} &\quad  \mbox{on }\Gamma_i, \ms\\
\ds \psi_e =12 \zeta_0\frac{\partial \varphi}{\partial \nu_e} &\quad  \mbox{on }\Gamma_e. \
\end{cases}
\end{align}
Furthermore, there exists a constant $C=C(\zeta_0,D,\Omega)$ such that
\begin{align*}
%\label{stability}
\|\psi_i\|_{L^2(\Gamma_i)}+\|\psi_e\|_{L^2(\Gamma_e )}
\leq C \big(\| \nabla P \|_{L^2(\Gamma_i)}+\left\|\nabla H \right\|_{L^2(\Gamma_e)}\big).
\end{align*}
\end{thm}

Based on Theorem \ref{well-posedness}, we obtain the representation formula of solution of the electro-osmosis model, which will be used to deduce the simultaneous electric and hydrodynamic cloaking conditions in the following sections.

\section{Simultaneous electric and hydrodynamic cloaking}\label{sec:simultaneous-cloaking}
This section focuses on  the proofs of Theorem \ref{main-thm-annulus} and \ref{main-thm-ellipses}, which provide the conditions for both electric and hydrodynamic cloaking to occur simultaneously.
We first examine the simultaneous electric and hydrodynamic cloaking by electro-osmosis when the boundaries $\p B$, $\p D$ and $\p \Omega$ are concentric circles and confocal ellipses in Subsections \ref{subsec:cloaking-annulus} and \ref{subsec:cloaking-ellipses}, respectively.  For these two special cases, we derive the explicit expression of the solution to \eqref{electro-osmotic equation}.  Subsection \ref{subsec:optimal} then uses an optimization method to explore simultaneous electric and hydrodynamic cloaking for a general shape.

\subsection{ Perfect electric and hydrodynamic cloaking on the annulus}\label{subsec:cloaking-annulus}
Throughout this subsection, we set $B :=\{|x| < r_o\}$,
$D :=\{|x| < r_i\}$ and $\Omega :=\{|x| < r_e\}$, where $r_e > r_i>r_o$.
For each integer $n$ and $a=o, i, e$, one can easily see in \cite{Ammari2013} that
\begin{align}\label{S-ra}
\mathcal{S}_{\Gamma}[\e^{\i n\theta}](x) = \begin{cases}
\ds -\frac{r_a}{2n}\Big( \frac{r}{r_a}\Big)^n \e^{\i n\theta},\quad & |x|=r  < r_a,\ms\\
\ds -\frac{r_a}{2n}\Big( \frac{r_a}{r}\Big)^n \e^{\i n\theta},\quad & |x|=r > r_a,
\end{cases}
\end{align}
and
\begin{equation}\label{K-ra}
\mathcal{K}_{\Gamma}^*[\e^{\i n\theta}](x)=0, \quad \forall \  n\neq 0.
\end{equation}

To begin with, we show the proof of Theorem \ref{main-thm-annulus}.
\renewcommand{\proofname}{Proof of  Theorem \ref{main-thm-annulus}}
\begin{proof}
Let $H(x) = r^n\e^{\i n\theta}$ for $n\geq 1$. From Theorem \ref{well-posedness} , we have
\begin{equation*}
\varphi=
H(x) + \mathcal{S}_{\Gamma_o}[\phi_o](x)+ \mathcal{S}_{\Gamma_i}[\phi_i](x),\quad x\in\mathbb{R}^2\setminus\overline{B},
\end{equation*}
where
\begin{align}
\label{varphi-density-equation-annulus}
\begin{cases}
\ds \Big(\frac{1}{2} I+\mathcal{K}^*_{\Gamma_o}\Big)[\phi_o]
+\frac{\partial\mathcal{S}_{\Gamma_i}[\phi_i]}{\partial\nu_o}= -\frac{\p H}{\p \nu_o} &\quad  \mbox{on }\Gamma_o, \ms\\
\ds \frac{\partial\mathcal{S}_{\Gamma_o}[\phi_o]}{\partial\nu_i} + \big(\lambda I + \mathcal{K}^*_{\Gamma_i}\big)[\phi_i] = -\frac{\p H}{\p \nu_i} &\quad  \mbox{on }\Gamma_i. \
\end{cases}
\end{align}
If $(\phi_o, \phi_i)$ is given by
\begin{align*}
\left[
  \begin{array}{ccc}
    \phi_{o} \ms \\
    \phi_{i}
  \end{array}
\right]
=
\left[
  \begin{array}{cc}
    \phi_{o}^n \ms \\
    \phi_{i}^n
  \end{array}
\right]
\e^{\i n\theta},
\end{align*}
then the integral equations \eqref{varphi-density-equation-annulus} are equivalent to
\begin{align*}
\begin{cases}
 \ds \frac{1}{2}\phi_o^n - \frac{1}{2}\big(\frac{r_o}{r_i}\big)^{n-1}\phi_i^{n} = -n r_o^{n-1}, \ms \\
 \ds \lambda\phi_i^n + \frac{1}{2}\big(\frac{r_o}{r_i}\big)^{n}\phi_o^n = -n r_i^{n-1},
\end{cases}
\end{align*}
By the  straightforward calculations using (\ref{K-ra}), we can obtain
\begin{align}\label{density-annulus-ep}
\begin{cases}
\ds \phi_o = - \frac{2n(1+2\lambda)r_i^{2n}r_o^{n-1}}{2\lambda r_i^{2n}+r_o^{2n}}\e^{\i n\theta},\ms \\
\ds \phi_i = \frac{2nr_i^{n-1}(r_o^{2n}-r_i^{2n})}{2\lambda r_i^{2n}+r_o^{2n}}\e^{\i n\theta}.
\end{cases}
\end{align}
Substituting (\ref{density-annulus-ep}) into (\ref{sol-ep}) and using (\ref{S-ra}), we can have the solution to \eqref{electro-osmotic equation}
\begin{align}\label{disk-varphi}
\varphi =
%\begin{cases}
%\ds \frac{(1+2\lambda)r_i^{2n}}{2\lambda r_i^{2n}+r_o^{2n}}\Big(r^n+\frac{r_o^{2n}}{r^n}\Big), &\quad r_o<r<r_i\ms \\
\ds r^n \e^{\i n\theta} + \frac{(2\lambda r_o^{2n}+r_i^{2n})r_i^{2n}}{2\lambda r_i^{2n}+r_o^{2n}}r^{-n}\e^{\i n\theta},&\quad r>r_i.
%\end{cases}
\end{align}

The electric cloaking condition $\varphi= H$ in $\R^2\setminus\overline{D}$ implies
\begin{align*}
2\lambda r_o^{2n}+r_i^{2n}=0,
\end{align*}
furthermore, we have
\begin{align*}
\varepsilon_s = \frac{r_i^{2n}+r_o^{2n}}{r_i^{2n}-r_o^{2n}}\varepsilon_m.
\end{align*}

Let $P(x)=12 r^n\e^{\i n\theta}$ for $n\geq 1$. By (\ref{density-equation}),  (\ref{S-ra}), (\ref{K-ra}), and (\ref{disk-varphi}), if
\begin{align*}
\left[
  \begin{array}{ccc}
    \psi_{i} \ms \\
    \psi_{e}
  \end{array}
\right]
=
\left[
  \begin{array}{cc}
    \psi_{i}^n \ms \\
    \psi_{e}^n
  \end{array}
\right]
\e^{\i n\theta},
\end{align*}
then we derive
\begin{align}
\label{density-annulus}
\begin{cases}
\ds \psi_i = -12n r_i^{n-1}(\zeta_0+2)\e^{\i n\theta},  \ms \\
\ds \psi_e = 12 n \zeta_0 r_e^{n-1} \e^{\i n\theta}.
\end{cases}
\end{align}
 Substituting (\ref{density-annulus}) into (\ref{sol-p}) and using (\ref{S-ra}), we can find that the solution to \eqref{electro-osmotic equation} is
 \begin{align}\label{annulus-p}
 p=
%  \begin{cases}
%\ds -6\Big((\zeta_0-2)r^n - (\zeta_0 +2)\frac{r_i^{2n}}{r^n})\Big)\e^{\i n\theta},\quad r_i<r<r_e,
%\vspace{1em}\\
\ds 12 r^n\e^{\i n\theta} - 6\Big((r_e^{2 n} - r_i^{2 n})\zeta_0 - 2r_i^{2n}\Big)\frac{1}{r^n}\e^{\i n\theta},\quad r>r_e.
%\end{cases}
 \end{align}

From the equation (\ref{annulus-p}), it follows that the outer flow and pressure satisfy the hydrodynamic cloaking conditions (\ref{cond-cloaking-p}) and (\ref{cond-cloaking}) provided
\begin{align*}
\zeta_0=\frac{2r_i^{2n}}{r_e^{2n}-r_i^{2n}}.
\end{align*}

The proof is complete.
\end{proof}

\subsection{Perfect electric and hydrodynamic cloaking on the confocal ellipses}\label{subsec:cloaking-ellipses}
In this subsection, we investigate the perfect electric and hydrodynamic cloaking to occur simultaneously on the confocal ellipses. For that purpose, we introduce elliptic coordinates.
The elliptic coordinates $(\xi, \eta)$ for $x=(x_1,x_2)$ in Cartesian coordinates are defined by
\begin{align*}
  x_1=l \cosh \xi \cdot \cos \eta, \quad x_2=l \sinh \xi \cdot \sin \eta,\quad \xi \geq 0, \quad 0\leq \eta \leq 2\pi,
\end{align*}
where $2l$ is the focal distance.
Suppose that $\p B=\Gamma_o$, $\p D=\Gamma_i$ and $\p \Omega=\Gamma_e$ are given by
$$
\Gamma_o = \{ (\xi, \eta) : \xi = \xi_o\}, \quad \Gamma_i = \{ (\xi, \eta) : \xi = \xi_i\} \quad \mbox{and} \quad \Gamma_e = \{ (\xi, \eta) : \xi = \xi_e\},
$$
where the number $\xi_o$, $\xi_i$ and $\xi_e$ are called the elliptic radius $\Gamma_o$, $\Gamma_i$ and $\Gamma_e$, respectively.

Let $\Gamma = \{ (\xi, \eta) : \xi = \xi_a\}$ for $a=o, i, e$. One can see easily that the length element $\d s$ and the outward normal derivative $\frac{\p}{\p \nu}$ on $\Gamma$ are given in terms of the elliptic coordinates by
\begin{equation*}
  \d s = \gamma \d \eta \quad \mbox{and} \quad \frac{\p}{\p \nu} = \gamma^{-1}\frac{\p}{\p \xi},
\end{equation*}
where
\begin{equation*}
\gamma =  \gamma (\xi_a, \eta) = l \sqrt{\sinh^2\xi_a+\sin^2\eta}.
\end{equation*}

To do that it is convenient to use the following notation: for $a=o, i, e$ and $n=1,2,\dots$,
\begin{align*}
  \beta_n^{c,a} := \gamma (\xi_a, \eta)^{-1} \cos (n\eta) \quad \mbox{and} \quad \beta_n^{s,a} := \gamma (\xi_a, \eta)^{-1} \sin (n\eta).
\end{align*}
For a nonnegative integer $n$ and $a=i, e$, it is proven in \cite{Chung2014, Ando2016} that
\begin{align}\label{S-ellipse-cos}
\mathcal{S}_{\Gamma}[ \beta_n^{c,a}](x) = \begin{cases}
\ds -\frac{\cosh (n\xi)}{n\e^{n\xi_a}}\cos (n\eta),\quad & \xi  < \xi_a,\ms\\
\ds -\frac{\cosh (n\xi_a)}{n\e^{n\xi}}\cos (n\eta),\quad & \xi  > \xi_a,
\end{cases}
\end{align}
and
\begin{align*}
%\label{S-ellipse-sin}
\mathcal{S}_{\Gamma}[\beta_n^{s,a}](x) = \begin{cases}
\ds -\frac{\sinh (n\xi)}{n\e^{n\xi_a}}\sin (n\eta),\quad & \xi  < \xi_a,\ms\\
\ds -\frac{\sinh (n\xi_a)}{n\e^{n\xi}}\sin (n\eta),\quad & \xi  > \xi_a.
\end{cases}
\end{align*}
Moreover, we also have
\begin{align}\label{K-ellipse}
\mathcal{K}^*_\Gamma [\beta_n^{c,a}] = \frac{1}{2 \e^{2n\xi_a}}\beta_n^{c,a}\quad \mbox{and} \quad \mathcal{K}^*_\Gamma [\beta_n^{s,a}] = -\frac{1}{2 \e^{2n\xi_a}}\beta_n^{s,a}.
\end{align}

We are ready to present the proof of Theorem \ref{main-thm-ellipses}.
%\begin{proof}[\indent Proof of  Theorem \ref{main-thm-ellipses}]
%  aa
%\end{proof}
\renewcommand{\proofname}{Proof of  Theorem \ref{main-thm-ellipses}}
\begin{proof}
Let $H(x)=\cosh (n\xi) \cos (n\eta)$ for $n\geq 1$. From the Theorem \ref{well-posedness} in Section \ref{sec:layer-potentials} and (\ref{K-ellipse}), if $(\phi_o, \phi_i)$ is given by
\begin{align*}
\left[
  \begin{array}{ccc}
    \phi_{i} \ms \\
    \phi_{e}
  \end{array}
\right]
=
\left[
  \begin{array}{cc}
    \phi_{o}^n \beta_n^{c,o} \ms \\
    \phi_{i}^n \beta_n^{c,i}
  \end{array}
\right],
\end{align*}
then the integral equations are equivalent to
\begin{align}\label{ep-density-ellipse}
\begin{cases}
 \ds \frac{\cosh(n\xi_o)}{\e^{n\xi_o}}\psi_i^n - \frac{\sinh(n\xi_o)}{\e^{n\xi_i}}\psi_e^{n} = -n \sinh(n\xi_o), \ms \\
 \ds \frac{\cosh(n\xi_o)}{\e^{n\xi_i}}\psi_i^n +\Big(\lambda+\frac{1}{2\e^{2n\xi_i}}\Big)\psi_e^n = -n \sinh(n\xi_i),
\end{cases}
\end{align}
It is readily seen that the solution to (\ref{ep-density-ellipse}) is given by
\begin{align}\label{density-ellipse-ep}
\begin{cases}
\ds \phi_o = - \frac{n(1+2\lambda)\e^{2n\xi_i}\e^{n\xi_o}\tanh(n\xi_o)}{2\lambda \e^{2n\xi_i}+\e^{2n\xi_o}}\beta_n^{c,o},\ms \\
\ds \phi_i =\frac{n\e^{n\xi_i}(\e^{2n\xi_o}-\e^{2n\xi_i})}{2\lambda \e^{2n\xi_i}+\e^{2n\xi_o}}\beta_n^{c,i}.
\end{cases}
\end{align}
 Substituting (\ref{density-ellipse-ep}) into (\ref{sol-ep}) and using (\ref{S-ellipse-cos}), we can have the solution to \eqref{electro-osmotic equation}
\begin{align}\label{ellipse-varphi}
\varphi =
%\begin{cases}
%\ds \frac{(1+2\lambda)\e^{2n\e^{n\xi_i}}}{2\lambda\e^{2n\xi_i}+\e^{2n\xi_o}}\big(\cosh (n\xi) + \e^{n\xi_o} \sinh (n\xi_o) \, \e^{-n\xi}\big)\cos (n\eta) \ms\\
\ds \Big(\cosh (n\xi) + \frac{(1+2\lambda)\sinh(n\xi_o)\e^{n(\xi_o+\xi_i)}-\cosh(n\xi_i)(\e^{2n\xi_o}-\e^{2n\xi_i})}{2\lambda\e^{2n\xi_i}+\e^{2n\xi_o}} \;\frac{\e^{n\xi_i}}{\e^{n\xi}}\Big)\cos (n\eta).
%\end{cases}
\end{align}
The cloaking condition $\varphi= H$ in $\R^2\setminus\overline{D}$ implies
\begin{align*}
(1+2\lambda)\sinh(n\xi_o)\e^{n(\xi_o+\xi_i)}-\cosh(n\xi_i)(\e^{2n\xi_o}-\e^{2n\xi_i})=0,
\end{align*}
furthermore, we have
\begin{align*}
  \varepsilon_s = \tanh(n\xi_i)\coth(n(\xi_i-\xi_o)).
\end{align*}

Let $P(x)=12 \cosh (n\xi) \cos (n\eta)$ for $n\geq 1$. From (\ref{density-equation}) and (\ref{ellipse-varphi}) we first have
\begin{equation}\label{ellipse-phi-e}
     \psi_{e} = 12 n \zeta_0\sinh (n\xi_e) \beta_n^{c,e}.
\end{equation}
If $ \psi_{i} =  \psi_{i}^n \beta_n^{c,i}$, then the first equation in (\ref{density-equation}) is equivalent to
\begin{equation*}
  \Big(\frac{1}{2} + \frac{1}{2 \e^{2n\xi_i}}\Big) \psi_i = 12 n\Big(\sinh (n\xi_e)\sinh (n\xi_i)\e^{-n\xi_e}\zeta_0- \sinh (n\xi_i)-\sinh(n\xi_i)\zeta_0 \Big)\beta_n^{c,i}.
\end{equation*}
By the simple calculation, we can obtain
\begin{equation}\label{ellipse-phi-i}
\psi_i = 12 n \e^{n\xi_i} \Big(\big(\sinh (n\xi_e)\e^{-n\xi_e}-1\big) \tanh (n\xi_i) \;\zeta_0- \tanh (n\xi_i) \Big)\beta_n^{c,i}.
\end{equation}
Substituting (\ref{ellipse-phi-e}) and (\ref{ellipse-phi-i}) into (\ref{sol-ep}) and using (\ref{S-ellipse-cos}), we can obtain the solution to \eqref{electro-osmotic equation}
{\small
 \begin{align}\label{ellipse-p-dir-x}
 p=
%  \begin{cases}
%\ds -12\Big((\sinh(n\xi_e)\e^{-n\xi_e}\zeta_0-1)\cosh(n\xi)+\Big(\big(\sinh (n\xi_e)\e^{-n\xi_e}-1\big)\sinh (n\xi_i) \;\zeta_0-  \sinh (n\xi_i)\Big)\frac{\e^{n\xi_i}}{\e^{n\xi}}\Big)\cos (n\eta),\ &\xi_i<\xi<\xi_e,
%\vspace{1em}\\
\ds \Big( 12 \cosh (n\xi) - 12\Big(\big(\sinh (n\xi_e) \cosh (n(\xi_e -\xi_i))- \sinh(n\xi_i) \big)\zeta_0  - \sinh (n\xi_i) \Big)\frac{\e^{n\xi_i}}{\e^{n\xi}}\Big)\cos (n\eta),\ &\xi>\xi_e.
%\end{cases}
 \end{align}}
 If $H(x) = \sinh (n\xi) \sin (n\eta)$ and $P = 12 \sinh (n\xi) \sin (n\eta)$ for $n\geq 1$, in a similar way we can find the solution to \eqref{electro-osmotic equation}
 \begin{equation}\label{ellipse-ep-dir-y}
  \varphi =\Big(\sinh (n\xi) + \frac{(1+2\lambda)\cosh(n\xi_o)\e^{n(\xi_o+\xi_i)}-\sinh(n\xi_i)(\e^{2n\xi_o}-\e^{2n\xi_i})}{2\lambda\e^{2n\xi_i}+\e^{2n\xi_o}} \;\frac{\e^{n\xi_i}}{\e^{n\xi}}\Big)\cos (n\eta),\quad \xi>\xi_i,
\end{equation}
 and the solution to \eqref{electro-osmotic equation}
 \begin{align}\label{ellipse-p-dir-y}
 p=
 \Big( 12 \sinh (n\xi) - 12\Big(\big(\cosh (n\xi_e) \cosh(n(\xi_e -\xi_i))- \cosh (n\xi_i)\big)\zeta_0  - \cosh (n\xi_i) \Big)\frac{\e^{n\xi_i}}{\e^{n\xi}}\Big)\sin (n\eta),\quad &\xi>\xi_e.
\end{align}
The cloaking conditions immediately follow from the equations \eqref{ellipse-p-dir-x}, (\ref{ellipse-ep-dir-y}) and (\ref{ellipse-p-dir-y}).

The proof is complete.
\end{proof}

\subsection{Approximate electric and hydrodynamic cloaking using optimization method}\label{subsec:optimal}
In this subsection, we propose a general framework for simultaneous electric and hydrodynamic cloaking. Based on Definition \ref{def-cloaking} and equation (\ref{cond-cloaking-p}), we only need to solve the following overdetermined boundary value problem with $\{B, D, \Omega; \varepsilon_s, \zeta_0\}$ to guarantee the simultaneous occurrence of the electric and hydrodynamic cloaking
\begin{align}\label{ep-cloaking}
\begin{cases}
\ds \Delta \varphi = 0  & \mbox{in } D\setminus\overline{B}, \ms \\
\ds \frac{\partial \varphi}{\partial \nu} = 0  & \mbox{on } \partial B, \ms \\
\ds  \varphi = H  & \mbox{on } \partial D,\ms \\
\ds  \varepsilon_s \frac{\partial \varphi}{\partial \nu} = \varepsilon_m\frac{\partial H}{\partial \nu}  & \mbox{on } \partial D,
\end{cases}
\end{align}
and $p$ satisfies:
\begin{align}\label{press-cloaking}
\begin{cases}
\ds \Delta p= 0 & \mbox{in }  \Omega\setminus\overline{D}, \ms\\
\ds \frac{\partial p}{\partial \nu} = -12 \zeta_0\frac{\partial \varphi}{\partial \nu}  & \mbox{on } \partial D, \ms\\
\ds p=P & \mbox{on }  \partial \Omega,\ms\\
\ds \frac{\partial P}{\partial \nu} - \frac{\partial p}{\partial \nu}  = 12 \zeta_0\frac{\partial \varphi}{\partial \nu} & \mbox{on } \partial \Omega.
\end{cases}
\end{align}

By using explicit expressions in Sections \ref{subsec:cloaking-annulus} and \ref{subsec:cloaking-ellipses}, we know that perfect electric and hydrodynamic cloaking occur simultaneously in two special instances. These cloaking structures $\{B,D,\Omega; \varepsilon_s,\zeta_0\}$ can be readily verified to satisfy equations (\ref{ep-cloaking}) and (\ref{press-cloaking}), respectively.
However, since the analytical solutions for the general shape do not exist, we can not prove the existence of the perfect cloaking structures $\{B,D,\Omega; \varepsilon_s,\zeta_0\}$ for the arbitrary shape. Specifically, there is no solution for the overdetermined problems (\ref{ep-cloaking}) and (\ref{press-cloaking}) with general domain.  Fortunately, the numerical experiments demonstrate that for the general shape, there is approximate electric and hydrodynamic cloaking.

To find the approximate electric cloaking structure $\{B, D;\varepsilon_s\}$ for general shape, we need to solve the following interior mixed boundary value problem
\begin{align}\label{cloaking-varphi-Dirichlet}
\begin{cases}
\ds \Delta \varphi = 0  & \mbox{in } D\setminus\overline{B}, \ms \\
\ds \frac{\partial \varphi}{\partial \nu} = 0  & \mbox{on } \partial B, \ms \\
\ds  \varphi = H  & \mbox{on } \partial D.
\end{cases}
\end{align}
Then we can obtain $\varepsilon_s$ satisfying cloaking condition (\ref{cond-cloaking-p}) by solving the following equation
\begin{equation*}
%\label{p-cloaking-zeta}
\ds \varepsilon_s \frac{\partial \varphi}{\partial \nu} - \varepsilon_m\frac{\partial H}{\partial \nu}=0  \quad \mbox{on }  \partial D.
\end{equation*}

In order to solve equation (\ref{cloaking-varphi-Dirichlet}), one can use the standard Nystr\"{o}m method or finite element method. It's worth noting that the solution to (\ref{cloaking-varphi-Dirichlet}) does not exist for the general shape. Hence we need to choose optimal $\varepsilon_s$ by optimization method.
Let the cost functional be
\begin{equation}\label{cost-functional-ep-cloaking}
  \mathcal{G}(\varepsilon_s)=\Big\|\ds  \varepsilon_s \frac{\partial \varphi}{\partial \nu} - \varepsilon_m\frac{\partial H}{\partial \nu}\Big\|^2_{L^2(\partial D)}, \quad \varepsilon_s \in [a_0, b_0],
\end{equation}
then the optimal permittivity is defined by
\begin{equation}\label{varepsilon-s-opt-cloaking}
  \varepsilon_{s,opt} :=\mathop{\arg\min}_{\varepsilon_s\in [a_0, b_0]} \,\mathcal{G}(\varepsilon_s).
\end{equation}

After the approximate electric cloaking occurs, in order to find the approximate hydrodynamic cloaking structure $\{D,\Omega;\zeta_0\}$ for general shape, we need to solve the following interior Neumann boundary value problem
\begin{align}\label{p-cloaking-Dirichlet}
\begin{cases}
\ds \Delta p= 0 & \mbox{in }  \Omega\setminus\overline{D}, \ms\\
\ds \frac{\partial p}{\partial \nu} = -12 \zeta_0\frac{\partial \varphi}{\partial \nu}  & \mbox{on } \partial D, \ms\\
\ds \frac{\partial p}{\partial \nu} = \frac{\partial P}{\partial \nu} - 12 \zeta_0\frac{\partial \varphi}{\partial \nu} &\mbox{on } \partial \Omega.
\end{cases}
\end{align}
Then we can obtain $\zeta_0$ satisfying cloaking condition (\ref{cond-cloaking-p}) by solving the following equation
\begin{equation*}
%\label{p-cloaking-zeta}
\ds p(\zeta_0) =P  \quad \mbox{on }  \partial \Omega.
\end{equation*}

Similarly, let the cost functional be
\begin{equation}\label{cost-functional-cloaking}
  \mathcal{F}(\zeta_0)=\left\|p(\zeta_0)-P\right\|^2_{L^2(\partial \Omega)}, \quad \zeta_0\in [c_0, d_0]\subset\R,
\end{equation}
then the optimal zeta potential is defined by
\begin{equation}\label{zeta-opt-cloaking}
  \zeta_{0,opt} :=\mathop{\arg\min}_{\zeta_0\in [c_0, d_0]} \,\mathcal{F}(\zeta_0).
\end{equation}
Namely, to design simultaneously the approximate electric and hydrodynamic cloaking, we can solve the PDE-constrained optimization problem \eqref{cloaking-varphi-Dirichlet}, \eqref{varepsilon-s-opt-cloaking}, (\ref{p-cloaking-Dirichlet}) and (\ref{zeta-opt-cloaking}).

The existence, uniqueness and stability of minimizer for the constrained optimization problem $\eqref{cost-functional-ep-cloaking}$ and $\eqref{cost-functional-cloaking}$ are given in the following theorems. Some proofs are similar, but for completeness, we still show more details on them.
\begin{thm}
There exists a unique optimal permittivity $\varepsilon_{s,opt} \in [a_0, b_0]$, which minimizes the cost functional $\mathcal{G}(\varepsilon_s)$ with PDE-constraint (\ref{cloaking-varphi-Dirichlet}) over all $\varepsilon_s\in [a_0, b_0]$.
\end{thm}
\renewcommand{\proofname}{Proof}
\begin{proof}
For any $\varepsilon_s^{(1)}$, $\varepsilon_s^{(2)}\in [a_0, b_0]$, by the $L^2$-estimation of solution for (\ref{cloaking-varphi-Dirichlet}),  we have
\begin{align*}
&|\mathcal{G}\big(\varepsilon_s^{(1)}\big) - \mathcal{G}\big(\varepsilon_s^{(2)}\big)| \\
%&= \Big|\Big\|\frac{\partial P}{\partial \nu} - \frac{\partial p}{\partial \nu}
%-12 \zeta_0^{(1)} \frac{\partial \varphi}{\partial \nu}\Big\|^2_{L^2(\partial \Omega)} - \Big\|\frac{\partial P}{\partial \nu} - \frac{\partial p}{\partial \nu}
%-12 \zeta_0^{(2)} \frac{\partial \varphi}{\partial \nu}\Big\|^2_{L^2(\partial \Omega)}\Big| \\
&\leq 2 \big|\varepsilon_s^{(1)}-  \varepsilon_s^{(2)}\big|\cdot \Big\|\varepsilon_m\frac{\p H}{\p \nu}\Big\|_{L^2(\partial D)} \cdot \Big\|\frac{\p \varphi}{\p\nu}\Big\|_{L^2(\partial D)} +  \big|\varepsilon_s^{(1)} -  \varepsilon_s^{(2)}\big|\cdot \big|\varepsilon_s^{(1)} +  \varepsilon_s^{(2)}\big|\cdot\Big\| \frac{\p \varphi}{\p \nu}\Big\|^2_{L^2(\partial D)}\\
&\leq C \big|\varepsilon_s^{(1)} -  \varepsilon_s^{(2)}\big|.
\end{align*}
Hence $\mathcal{G}(\varepsilon_s)$ is Lipschitz continuous in $[a_0, b_0]$. Furthermore, supposing $\lambda_1$, $\lambda_2 \in (0, 1)$ and $\lambda_1+\lambda_2=1$, we obtain
\begin{align*}
&\mathcal{G}\big(\lambda_1\varepsilon_s^{(1)}+\lambda_2 \varepsilon_s^{(2)}\big)\\
&= \Big\|\big(\lambda_1\varepsilon_s^{(1)}+\lambda_2 \varepsilon_s^{(2)}\big) \frac{\partial \varphi}{\partial \nu}-\varepsilon_m\frac{\partial H}{\partial \nu} \Big\|^2_{L^2(\partial D)}\\
%&=\Big\|(\lambda_1+\lambda_2)\frac{\partial P}{\partial \nu} - (\lambda_1+\lambda_2)\frac{\partial p}{\partial \nu}
%-12 \big(\lambda_1\zeta_0^{(1)}+\lambda_2 \zeta_0^{(2)}\big) \frac{\partial \varphi}{\partial \nu}\Big\|^2_{L^2(\partial \Omega)}\\
%&\leq \Big(\lambda_1\Big\|\frac{\partial P}{\partial \nu} - \frac{\partial p}{\partial \nu}
%-12 \zeta_0^{(1)} \frac{\partial \varphi}{\partial \nu}\Big\|_{L^2(\partial \Omega)} + \lambda_2\Big\|\frac{\partial P}{\partial \nu} - \frac{\partial p}{\partial \nu}
%-12 \zeta_0^{(2)} \frac{\partial \varphi}{\partial \nu}\Big\|_{L^2(\partial \Omega)}\Big)^2\\
&< \lambda_1\Big\|\varepsilon_s^{(1)} \frac{\partial \varphi}{\partial \nu}-\varepsilon_m\frac{\partial H}{\partial \nu}\Big\|^2_{L^2(\partial \Omega)} + \lambda_2\Big\|\varepsilon_s^{(2)} \frac{\partial \varphi}{\partial \nu}-\varepsilon_m\frac{\partial H}{\partial \nu}\Big\|^2_{L^2(\partial \Omega)}\\
&= \lambda_1\mathcal{G}\big(\varepsilon_s^{(1)}\big)+ \lambda_2 \mathcal{G}\big(\varepsilon_s^{(2)}\big).
\end{align*}
Then $\mathcal{G}(\varepsilon_s)$ is convex strictly in $[a_0, b_0]$.  Therefore the cost functional \eqref{cost-functional-ep-cloaking} has a unique minimizer.

The proof is complete.
\end{proof}

The following  decomposition lemma  presents an explicit relationship between $p$ and $\zeta_0$, which  is of critical importance for our subsequent analysis.
\begin{lem}\label{lem:decomposition}
  The solution to \eqref{p-cloaking-Dirichlet} can be represented as
  \begin{align*}
  %\label{decomposition}
    p = p_1 + \zeta_0 p_2  \qquad \mbox{in }\  \Omega\setminus\overline{D}
  \end{align*}
  where $p_1$ and $p_2$ are independent on $\zeta_0$.
\end{lem}
\begin{proof}
  Let the solution to \eqref{p-cloaking-Dirichlet} be represented as
  \begin{align}\label{p-cloaking-Dirichlet-layer-potential}
   p = \mathcal{S}_{\Gamma_i}[\psi_i](x) + \mathcal{S}_{\Gamma_e}[\psi_e](x),\quad  x\in  \Omega\setminus\overline{D},
  \end{align}
  where the density function pair $(\phi_o, \phi_i)\in L_0^2(\Gamma_o)\times L_0^2(\Gamma_i)$ satisfy
  \begin{align}
\label{p-cloaking-Dirichlet-density-equation}
\begin{cases}
\ds \Big(\frac{1}{2} I+\mathcal{K}^*_{\Gamma_i}\Big)[\psi_i]
+\frac{\partial\mathcal{S}_{\Gamma_e}[\psi_e]}{\partial\nu_i}= -12 \zeta_0\frac{\partial \varphi}{\partial \nu_i} &\quad  \mbox{on }\Gamma_i, \ms\\
\ds \frac{\partial\mathcal{S}_{\Gamma_i}[\psi_i]}{\partial\nu_e} + \Big(-\frac{1}{2} I + \mathcal{K}^*_{\Gamma_e}\Big)[\psi_e] = \frac{\partial P}{\partial \nu_e} - 12 \zeta_0\frac{\partial \varphi}{\partial \nu_e} &\quad  \mbox{on }\Gamma_i. \
\end{cases}
\end{align}
Denote the Neumann–Poincaré-type operator $\mathbb{K}^*$, $\Psi$ and $g$ by
\begin{align*}
  \mathbb{K}^*:=\left[
  \begin{array}{cc}
 \ds  \mathcal{K}^*_{\Gamma_i} &\ds \frac{\partial}{\partial\nu_i}\mathcal{S}_{\Gamma_e} \ms \\
 \ds  -\frac{\partial}{\partial\nu_e}\mathcal{S}_{\Gamma_i} &  \ds -\mathcal{K}^*_{\Gamma_e}
  \end{array}
  \right],\quad
  \Psi := \left[
  \begin{array}{c}
    \psi_i \ms\\
    \psi_e
  \end{array}
  \right],\quad
  g:=  \left[
  \begin{array}{c}
 \ds -12 \zeta_0\frac{\partial \varphi}{\partial \nu_i}\ms \\
  \ds -\frac{\partial P}{\partial \nu_e} + 12 \zeta_0\frac{\partial \varphi}{\partial \nu_e}
  \end{array}
  \right]
 \end{align*}
 Then, \eqref{p-cloaking-Dirichlet-density-equation} can be rewritten in the form
 \begin{align*}
   \Big(\frac{1}{2}  \mathbb{I}+\mathbb{K}^*\Big)\Psi =g,
 \end{align*}
 where $\mathbb{I}$ is given by
 \begin{align*}
   \mathbb{I}:= \left[
  \begin{array}{cc}
   I & 0 \ms \\
   0 &  I
  \end{array}
  \right].
  \end{align*}

 We now decompose the right-hand side $g$ and let
 \begin{align*}
  g_1=  \left[
  \begin{array}{c}
 \ds 0\ms \\
  \ds -\frac{\partial P}{\partial \nu_e}
  \end{array}
  \right],\quad
    g_2=  \left[
  \begin{array}{c}
 \ds -12 \frac{\partial \varphi}{\partial \nu_i}\ms \\
  \ds  12 \frac{\partial \varphi}{\partial \nu_e}
  \end{array}
  \right],
 \end{align*}
 then $g =g_1 + \zeta_0 g_2$. Furthermore, since $\big(\frac{1}{2}  \mathbb{I}+\mathbb{K}^*\big)$ is invertible, it follows that $\Psi = \Psi_1 + \zeta_0 \Psi_2$, that is,
\begin{align}\label{p-cloaking-Dirichlet-density-decomposition}
\psi_i = \psi_{i,1} +\zeta_0 \psi_{i,2} \quad \mbox{and } \quad \psi_e = \psi_{e,1} +\zeta_0 \psi_{e,2}.
\end{align}
Substituting \eqref{p-cloaking-Dirichlet-density-decomposition} into \eqref{p-cloaking-Dirichlet-layer-potential}, we have
\begin{align*}
p = \mathcal{S}_{\Gamma_i}[\psi_{i,1}](x) + \mathcal{S}_{\Gamma_e}[\psi_{e,1}](x) + \zeta_0\big(\mathcal{S}_{\Gamma_i}[\psi_{i,2}](x) + \mathcal{S}_{\Gamma_e}[\psi_{e,2}](x)\big),\quad  x\in  \Omega\setminus\overline{D}.
\end{align*}
Let $p_1=\mathcal{S}_{\Gamma_i}[\psi_{i,1}](x) + \mathcal{S}_{\Gamma_e}[\psi_{e,1}](x)$ and $p_2=\mathcal{S}_{\Gamma_i}[\psi_{i,2}](x) + \mathcal{S}_{\Gamma_e}[\psi_{e,2}](x)$ in $\Omega\setminus\overline{D}$, then $p=p_1+\zeta_0 p_2$.

The proof is complete.
\end{proof}

\begin{thm}
There exists a unique optimal zeta potential  $\zeta_{0,opt} \in [c_0, d_0]$, which minimizes the cost functional $\mathcal{F}(\zeta_0)$ with PDE-constraint (\ref{p-cloaking-Dirichlet}) over all $\zeta_0\in [c_0, d_0]$.
\end{thm}
\renewcommand{\proofname}{Proof}
\begin{proof}
For any $\zeta_0^{(1)}$, $\zeta_0^{(2)}\in [c_0, d_0]$, by the $L^2$-estimation of solution for (\ref{p-cloaking-Dirichlet}), from Lemma \ref{lem:decomposition} we have
\begin{align*}
&|\mathcal{F}\big(\zeta_0^{(1)}\big) - \mathcal{F}\big(\zeta_0^{(2)}\big)| \\
&\leq 2 \big|\zeta_0^{(1)}-  \zeta_0^{(2)}\big|\cdot \|p_1 - P\|_{L^2(\partial \Omega)} \cdot \|p_2\|_{L^2(\partial \Omega)} + \big|\zeta_0^{(1)} -  \zeta_0^{(2)}\big|\cdot \big|\zeta_0^{(1)} +  \zeta_0^{(2)}\big|\cdot \|p_2\|^2_{L^2(\partial \Omega)}\\
&\leq C \big|\zeta_0^{(1)} -  \zeta_0^{(2)}\big|.
\end{align*}
Hence $\mathcal{F}(\zeta_0)$ is Lipschitz continuous in $[a_0, b_0]$. Furthermore, supposing $\lambda_1$, $\lambda_2 \in (0, 1)$ and $\lambda_1+\lambda_2=1$, we obtain
\begin{align*}
&\mathcal{F}\Big(\lambda_1\zeta_0^{(1)}+\lambda_2 \zeta_0^{(2)}\Big)\\
&= \big\|p_1 +\big(\lambda_1\zeta_0^{(1)}+\lambda_2 \zeta_0^{(2)}\big)p_2-P\big\|^2_{L^2(\partial \Omega)}\\
&< \lambda_1\big\|p_1 + \zeta_0^{(1)} p_2-P\Big\|^2_{L^2(\partial \Omega)} + \lambda_2\Big\|p_1 + \zeta_0^{(2)} p_2 - P \big\|^2_{L^2(\partial \Omega)}\\
&= \lambda_1\mathcal{F}\big(\zeta_0^{(1)}\big)+ \lambda_2 \mathcal{F}\big(\zeta_0^{(2)}\big).
\end{align*}
Then $\mathcal{F}(\zeta_0)$ is convex strictly in $[c_0, d_0]$.  Therefore the cost functional \eqref{cost-functional-cloaking} has a unique minimizer.

The proof is complete.
\end{proof}

We next show some stability results.
\begin{thm}
Let $\big\{\varphi^{(k)}\big\}$ be sequences such that $\varphi^{(k)}\rightarrow \varphi$, as $k \rightarrow \infty$ in $L^2(\partial D)$ and let $\{\varepsilon_{s,opt}^{(k)}\}$ be a minimizer of \eqref{cost-functional-ep-cloaking} with $\varphi$ replaced by $\varphi_k$. Then the optimal permittivity is stable with respect to the electrostatic potential $\varphi$, i.e., ${\varepsilon_{s,opt}^{(k)}}\rightarrow\varepsilon_{s,opt}$.
\end{thm}
\begin{proof}
  From the definition of ${\varepsilon_{s,opt}^{(k)}}$, we find
  \begin{align}\label{ineq}
  \Big\|\varepsilon_{s,opt}^{(k)} \frac{\partial \varphi^{(k)}}{\partial \nu} - \varepsilon_m\frac{\partial H}{\partial \nu}\Big\|^2_{L^2(\partial D)}
  \leq\Big\|\varepsilon_{s} \frac{\partial \varphi^{(k)}}{\partial \nu} - \varepsilon_m\frac{\partial H}{\partial \nu}\Big\|^2_{L^2(\partial D)}, \quad \forall \;  \zeta_0\in [a_0, b_0].
  \end{align}
 Since $\varepsilon_{s,opt}^{(k)}\in [a_0, b_0]$, there exists a subsequence, still denoted $\{\varepsilon_{s,opt}^{(k)}\}$, which implies $\varepsilon_{s,opt}^{(k)}\rightarrow\varepsilon_{s,opt}^{(0)}$  and $\varepsilon_{s,opt}\in [a_0, b_0]$.  Based on the continuity of $L^2$-norm, by \eqref{ineq} it follows that
  \begin{align*}
  \left\|\varepsilon_{s,opt}^{(0)} \frac{\partial \varphi}{\partial \nu} - \varepsilon_m\frac{\partial H}{\partial \nu}\right\|^2_{L^2(\partial D)}
  &= \lim_{k\rightarrow\infty}\left\|\varepsilon_{s,opt}^{(k)} \frac{\partial \varphi^{(k)}}{\partial \nu} - \varepsilon_m\frac{\partial H}{\partial \nu}\right\|^2_{L^2(\partial D)}\\
  &\leq \lim_{k\rightarrow\infty}\left\|\varepsilon_{s} \frac{\partial \varphi^{(k)}}{\partial \nu} - \varepsilon_m\frac{\partial H}{\partial \nu}\right\|^2_{L^2(\partial D)}\\
  &=  \left\|\varepsilon_{s} \frac{\partial \varphi}{\partial \nu} - \varepsilon_m\frac{\partial H}{\partial \nu}\right\|^2_{L^2(\partial D)},
  \end{align*}
  for  all $\varepsilon_s\in [a_0, b_0]$. This deduces that $\varepsilon_{s,opt}^{(0)}$ is a minimizer of $\mathcal{G}(\varepsilon_s)$. By the uniqueness of the minimizer, we have $\varepsilon_{s,opt}^{(0)}=\varepsilon_{s,opt}$.
  The proof is complete.
\end{proof}

\begin{thm}
Let $\big\{p_1^{(k)}\big\}$ and $\{p_2^{(k)}\}$ be sequences such that $p_1^{(k)}\rightarrow p_1$ and $p_2^{(k)} \rightarrow p_2$, as $k \rightarrow \infty$ in $L^2(\partial \Omega)$ and let $\{\zeta_{0,opt}^{(k)}\}$ be a minimizer of \eqref{cost-functional-cloaking} with $\varphi$ and $p$ replaced by
$\varphi_k$ and $p_k$. Then the optimal zeta potential is stable with respect to the electrostatic potential $\varphi$ and pressure $p$, i.e., ${\zeta_{0,opt}^{(k)}}\rightarrow\zeta_{0,opt}$.
\end{thm}
\begin{proof}
  From Lemma \ref{lem:decomposition} and the definition of ${\zeta_{0,opt}^{(k)}}$, we find
  \begin{align}\label{ineq}
  \Big\|p_{1}^{(k)} + \zeta_{0,opt}^{(k)} \; p_{2}^{(k)} - P\Big\|^2_{L^2(\partial \Omega)}
  \leq\Big\|p_{1}^{(k)} + \zeta_0 \; p_{2}^{(k)} - P\Big\|^2_{L^2(\partial \Omega)}, \quad \forall \;  \zeta_0\in [c_0, d_0].
  \end{align}
 Since $\zeta_{0,opt}^{(k)}\in [c_0, d_0]$, there exists a subsequence, still denoted $\{\zeta_{0,opt}^{(k)}\}$, which implies $\zeta_{0,opt}^{(k)}\rightarrow\zeta_{0,opt}^{(0)}$  and $\zeta_{0,opt}\in [c_0, d_0]$.  Based on the continuity of $L^2$-norm, by \eqref{ineq} it follows that
  \begin{align*}
  \left\|p_{1} + \zeta_{0,opt}^{(0)} \; p_{2} - P\right\|^2_{L^2(\partial \Omega)}
  &= \lim_{k\rightarrow\infty}\left\|p_{1}^{(k)} + \zeta_{0,opt}^{(k)} \; p_{2}^{(k)} - P\right\|^2_{L^2(\partial \Omega)}\\
  &\leq \lim_{k\rightarrow\infty}\left\|p_{1}^{(k)} + \zeta_0 \; p_{2}^{(k)} - P\right\|^2_{L^2(\partial \Omega)}\\
  &=  \left\|p_{1} + \zeta_0 \; p_{2} - P\right\|^2_{L^2(\partial \Omega)},
  \end{align*}
  for  all $\zeta_0\in [c_0, d_0]$. This deduces that $\zeta_{0,opt}^{(0)}$ is a minimizer of $\mathcal{F}(\zeta_0)$. By the uniqueness of the minimizer, we have $\zeta_{0,opt}^{(0)}=\zeta_{0,opt}$.
  The proof is complete.
\end{proof}
\begin{rem}
Indeed, we can infer the stability of minimizer with respect to  background fields from the continuous dependence on background fields $H$ and $P$ of the solutions of (\ref{cloaking-varphi-Dirichlet}) and (\ref{p-cloaking-Dirichlet}).
\end{rem}

Finally, we give the proof of Theorem \ref{thm-estimation}.
\renewcommand{\proofname}{\indent Proof of  Theorem \ref{thm-estimation}}
\begin{proof}
We first prove the case of electric cloaking.
From the boundary condition $\varepsilon_m\frac{\partial \varphi}{\partial \nu} \big|_{+} = \varepsilon_s \frac{\partial \varphi}{\partial \nu} \big|_{-} $  on $\partial D$ defined in
equation \eqref{electro-osmotic equation}, we derive
\begin{align}\label{eq-boundary-electric}
\varepsilon_m\frac{\partial \varphi}{\partial \nu} \Big|_{+} - \varepsilon_m \frac{\partial H}{\partial \nu} = \varepsilon_s \frac{\partial \varphi}{\partial \nu} \Big|_{-} - \varepsilon_m \frac{\partial H}{\partial \nu} \quad \mbox{on } \p D.
\end{align}
Applying the $L^2(\partial D)$ norm to both sides of the equation \eqref{eq-boundary-electric} gives
\begin{align*}
 \Big\|\varepsilon_m\frac{\partial \varphi}{\partial \nu} \Big|_{+} - \varepsilon_m \frac{\partial H}{\partial \nu} \Big\|_{L^2(\partial D)} =\Big\| \varepsilon_s \frac{\partial \varphi}{\partial \nu} \Big|_{-} - \varepsilon_m \frac{\partial H}{\partial \nu}\Big\|_{L^2(\partial D)}.
\end{align*}
Setting $\varepsilon_s=\varepsilon_{s,opt}$ and using the condition $\mathcal{G}(\varepsilon_{s,opt})<\epsilon^2$, we obtain $\Big\| \frac{\partial \varphi}{\partial \nu} \big|_{+} -  \frac{\partial H}{\partial \nu} \Big\|_{L^2(\partial D)} < \epsilon$.
Applying Green's formula \cite{Kress1989} to $H$ and $\varphi$ in $\R^2 \setminus \overline{D}$, we have
\begin{alignat}{2}\label{Gf-exterior1-electric}
\ds H(x)&=H_{\infty} + \int_{\p D} G(x,y) \frac{\p H}{\p \nu}(y)\d s(y) - \int_{\p D} \frac{\p G(x,y)}{\p \nu(y)}H(y)\d s(y), \quad &&x\in \R^2 \setminus \overline{D},\\
\ds \varphi(x)&=\varphi_{\infty} +  \int_{\p D} G(x,y) \frac{\p \varphi}{\p \nu}(y)\d s(y) - \int_{\p D} \frac{\p G(x,y)}{\p \nu(y)}\varphi(y)\d s(y), \quad &&x\in \R^2\setminus \overline{D}, \label{Gf-exterior2-electric}
\end{alignat}
where the mean value property at infinity
\begin{align*}
H_{\infty} = \frac{1}{2 \pi r} \int_{|y|=r} H(y) \d s(y) \quad \mbox{and} \quad \varphi_{\infty} = \frac{1}{2 \pi r} \int_{|y|=r} \varphi(y) \d s(y)
\end{align*}
for sufficiently large $r$ are satisfied.
In the case of cloaking, we require  $\varphi|_{+}=H$ on $\p D$. Moreover, from \eqref{Gf-exterior1-electric}, \eqref{Gf-exterior2-electric} and the boundary condition $\Big\|\frac{\partial \varphi}{\partial \nu} \big|_{+} -  \frac{\partial H}{\partial \nu}\Big\|_{L^2(\partial D)} < \epsilon$, the following inequalities hold.
\begin{align*}
|\varphi(x) - H(x)|&\leq |\varphi_{\infty} - H_{\infty}| + \Big|\int_{\p D} G(x,y) \Big(\frac{\p \varphi}{\p \nu}(y)-\frac{\p H}{\p \nu}(y)\Big)\d s(y) - \int_{\p D} \frac{\p G(x,y)}{\p \nu(y)}\big(\varphi(y) - H(y)\big)\d s(y)\Big|\\
&\leq \frac{1}{2 \pi r} \int_{|y|=r} |\varphi(y)- H(y)| \d s(y) + \Big(\int_{\p D} G^2(x,y)\d s(y)\Big)^{\frac{1}{2}} \cdot \Big(\int_{\p D}\Big(\frac{\p \varphi}{\p \nu}(y)-\frac{\p H}{\p \nu}(y)\Big)^2\d s(y)\Big)^{\frac{1}{2}}\\
&\leq C\Big\|\frac{\partial \varphi}{\partial \nu} \Big|_{+} -  \frac{\partial H}{\partial \nu} \Big\|_{L^2(\partial D)}\\
&<\epsilon, \quad x\in \R^2 \setminus \overline{D}.
\end{align*}

We next prove the case of hydrodynamic clokaing.
From the boundary condition $ p|_{+} =  p|_{-}$  on $\partial \Omega$ defined in
equation \eqref{electro-osmotic equation}, we derive
\begin{align}\label{eq-boundary}
p|_{+} - P = p|_{-} -P \quad \mbox{on } \p \Omega.
\end{align}
Applying the $L^2(\partial \Omega)$ norm to both sides of the equation \eqref{eq-boundary} gives
\begin{align*}
 \big\|p|_{+} - P  \big\|_{L^2(\partial \Omega)} =\big\| p|_{-} - P \big\|_{L^2(\partial \Omega)}.
\end{align*}
Setting $\zeta_0=\zeta_{0,opt}$ and using the condition $\mathcal{F}(\zeta_{0,opt})<\epsilon^2$, we obtain $\big\| p |_{+} - P \big\|_{L^2(\partial \Omega)} < \epsilon$.
Applying Green's formula \cite{Kress1989} to $P$ and $p$ in $\R^2 \setminus \overline{\Omega}$, we have
\begin{alignat}{2}\label{Gf-exterior1}
\ds P(x)&=P_{\infty} + \int_{\p \Omega} G(x,y) \frac{\p P}{\p \nu}(y)\d s(y) - \int_{\p \Omega} \frac{\p G(x,y)}{\p \nu(y)}P(y)\d s(y), \quad &&x\in \R^2 \setminus \overline{\Omega},\\
\ds p(x)&=p_{\infty} +  \int_{\p \Omega} G(x,y) \frac{\p p}{\p \nu}(y)\d s(y) - \int_{\p \Omega} \frac{\p G(x,y)}{\p \nu(y)}p(y)\d s(y), \quad &&x\in \R^2\setminus \overline{\Omega}, \label{Gf-exterior2}
\end{alignat}
where the mean value property at infinity
\begin{align*}
P_{\infty} = \frac{1}{2 \pi r} \int_{|y|=r} P(y) \d s(y) \quad \mbox{and} \quad p_{\infty} = \frac{1}{2 \pi r} \int_{|y|=r} p(y) \d s(y)
\end{align*}
for sufficiently large $r$ are satisfied.
In the case of hydrodynamic cloaking, we require  $\frac{\partial p}{\partial \nu} \big|_{+}=\frac{\partial P}{\partial \nu}$ on $\p \Omega$. Moreover, from \eqref{Gf-exterior1}, \eqref{Gf-exterior2} and the boundary condition $\big\| p|_{+} - P   \big\|_{L^2(\partial \Omega)} < \epsilon$, the following inequalities hold.
\begin{align*}
|p(x) - P(x)|&\leq |p_{\infty} - P_{\infty}| + \Big|\int_{\p \Omega} G(x,y) \Big(\frac{\p p}{\p \nu}(y)-\frac{\p P}{\p \nu}(y)\Big)\d s(y) - \int_{\p \Omega} \frac{\p G(x,y)}{\p \nu(y)}\big(p(y) - P(y)\big)\d s(y)\Big|\\
&\leq \frac{1}{2 \pi r} \int_{|y|=r} |p(y)- P(y)| \d s(y) + \Big(\int_{\p \Omega} G^2(x,y)\d s(y)\Big)^{\frac{1}{2}} \cdot \Big(\int_{\p \Omega}( p(y)- P(y))^2\d s(y)\Big)^{\frac{1}{2}}\\
&\leq C\big\|  p |_{+} - P \big\|_{L^2(\partial \Omega)}\\
&<\epsilon, \quad x\in \R^2 \setminus \overline{\Omega}.
\end{align*}

The proof is complete.
\end{proof}

The following corollary can be found in the proof of Theorem \ref{thm-estimation}.
\begin{cor}\label{cor} Let $p$ be the solution to (\ref{electro-osmotic equation}) with $\varphi|_{+}=H$ on $\p D$. Then there exists a positive constant $C$ such that
\begin{align*}
|\varphi(x) - H(x)|\leq C\Big\|\frac{\partial \varphi}{\partial \nu} \Big|_{+} -  \frac{\partial H}{\partial \nu} \Big\|_{L^2(\partial D)}, \quad x\in \R^2 \setminus \overline{D}.
\end{align*}
 Let $p$ be the solution to (\ref{electro-osmotic equation}) with $\frac{\p p}{\p \nu}|_{-}=\frac{\p P}{\p \nu}$ on $\p \Omega$. Then there exists a positive constant $C$ such that
\begin{align*}
|p(x)-P(x)|\leq C\big\|  p |_{+} - P \big\|_{L^2(\partial \Omega)}, \quad x\in \R^2 \setminus \overline{\Omega}.
\end{align*}
\end{cor}
It can be seen from Corollary \ref{cor} that the errors are minimized when the cost functionals $\mathcal{G}(\zeta_0)$ and $\mathcal{F}(\zeta_0)$ are minimized. We shall verify it with many numerical examples given in Section \ref{sec:NumSim}.

\section{Numerical simulations}\label{sec:NumSim}
In this section, we perform three-dimensional finite-element simulations to validate the theoretical results based on a reduced two-dimensional model with the help of the commercial software COMSOL Multiphysics. The results show good agreement.

Before showing the numerical results, we provide the values of physical and geometrical parameters used in the finite-element simulations, which are summarized in the following table.
\begin{table}[H]
  \centering
  \begin{tabular}{cccc}
\centering
  Physical property & Notation & Value &Units \\
  \hline
  Gap between the plates & $\tilde{h}$ & 15 & $\mu \mathrm{m}$ \\
  Length of computational domain & $\tilde{L}$ & 2 &$\mathrm{mm}$\\
  Width of computational domain& $\tilde{W}$  & 2&$\mathrm{mm}$\\
  Density of fluid & $\tilde{\rho}$ & $10^{3}$ &$\mathrm{Kg}/\mathrm{m}^{3}$\\
  Viscosity of fluid & $\tilde{\mu}$ & $10^{-3}$ &$\mathrm{Pa} \cdot \mathrm{s}$\\
  Permittivity of fluid  & $\tilde{\varepsilon}_m$ & $7.08 \times 10^{-10}$ & $\mathrm{F}/\mathrm{m}$\\
  Electric field far from object& $\tilde{E}$ & $3\times 10^{2}  $ & $\mathrm{V}/\mathrm{m}$\\
  External velocity & $\tilde{u}_{ext}$ & $51$ & $\mu \mathrm{m} / \mathrm{s}$\\
  \hline
\end{tabular}
  \caption{Values of physical and geometrical parameters used in the finite-element simulations.}
\end{table}

We first perform finite-element simulations of the flow around a circular cylinder of radius $100\  \mu \mathrm{m}$ in a Hele-Shaw cell, in which the circular cylinder consists of a circular core of radius $50\  \mu \mathrm{m}$ and a cloaking shell.  The radius of the hydrodynamic cloaking region is $200\  \mu \mathrm{m}$.
Figure \ref{fig:circle} presents a comparison of finite-element simulation results corresponding to an electric field without cloak (a), pressure-driven flow (c) and cloaking (b,d) conditions. Figures \ref{fig:circle}(a)--\ref{fig:circle}(b) present the resulting electrostatic potential distribution (colormap), equipotential lines (black lines) and electric field lines (white lines) without and with an electric cloak design, respectively. As shown in Figure \ref{fig:circle}(a) for the controlling sample, the insulated core causes a strong disturbance of the electrostatic potential and forms a field "shadow" after it. This disturbing effect is completely suppressed by applying a $50 \  \mu \mathrm{m} $--thick electric cloaking shell, as evidenced by the straight potential contour lines right outside the cloak as shown in Figure \ref{fig:circle}(b), showing excellent electric cloaking in the finite-element simulations for the circular cylinder core coated by the shell. Figures \ref{fig:circle}(c)--\ref{fig:circle}(d) present the resulting  pressure distribution (colormap), equipressure lines (black lines) and streamlines (white lines) without and with a hydrodynamic cloak respectively. In Figure \ref{fig:circle}(c),
the no-penetration object causes a pressure shadow on its right side. With a desired  hydrodynamic cloak, as shown in Figure \ref{fig:circle}(d), this disturbance of the flow is totally diminished, as evidenced by the straight equipressure lines, showing excellent hydrodynamic cloaking in the finite-element simulations for the circular cylinder object. Under cloaking conditions (\ref{annulus-cloaking-zeta}), the electric field lines and streamlines outside of the control region are straight, unmodified relative to the uniform far field, and undisturbed by the core and object.
\begin{figure}[H]
	\centering  %图片全局居中
	\subfigbottomskip=0pt %两行子图之间的行间距
	\subfigcapskip=-10pt %设置子图与子标题之间的距离
	\subfigure[]{
		\includegraphics[width=0.23\linewidth]{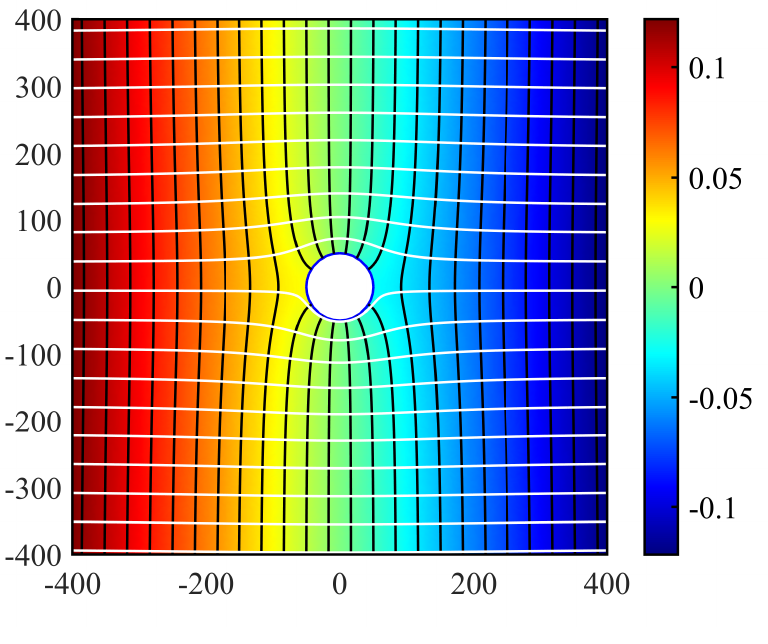}}
	\subfigure[]{
		\includegraphics[width=0.23\linewidth]{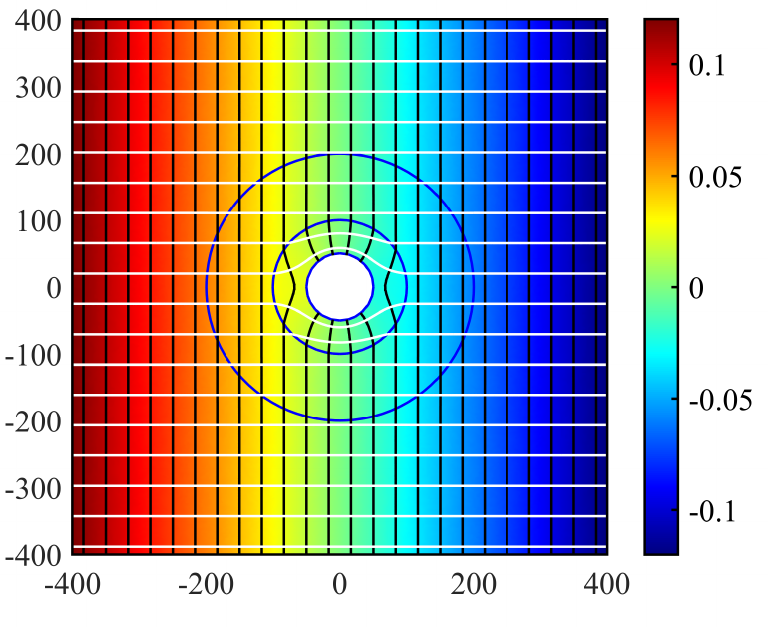}}
	\subfigure[]{
		\includegraphics[width=0.23\linewidth]{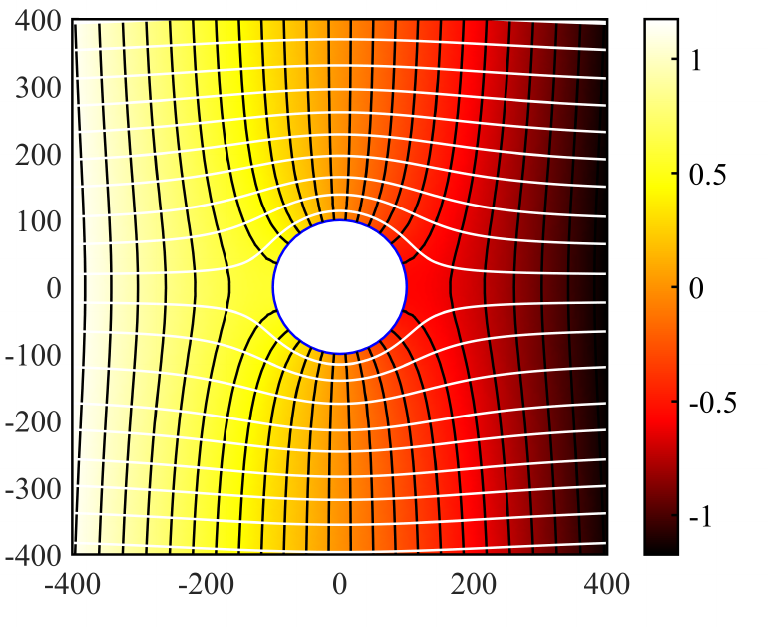}}
	\subfigure[]{
		\includegraphics[width=0.23\linewidth]{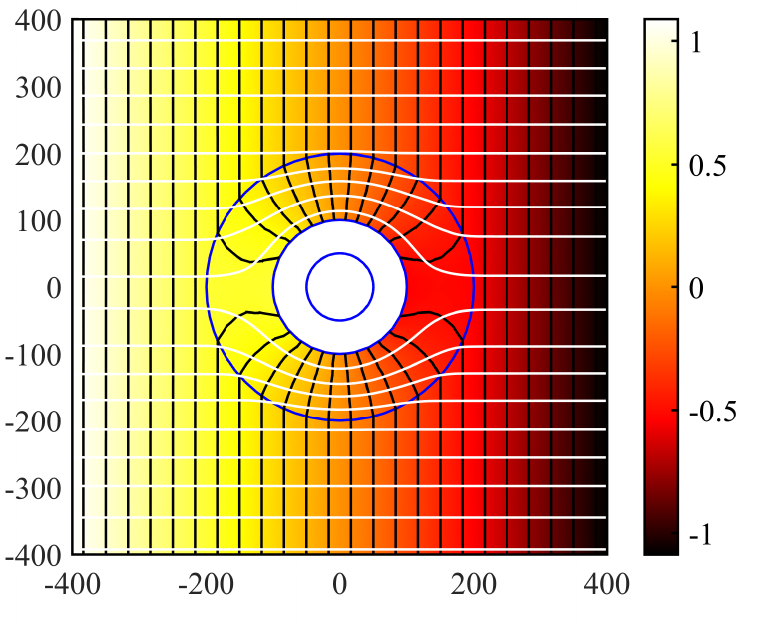}}
%	\subfigure[]{
%		\includegraphics[width=0.32\linewidth]{cloaking-circle.pdf}}
	\caption{Comparison of finite-element simulation results on the annulus. Numerical (a,b) results for the electrostatic potential distribution (colormap),  equipotential lines (black lines) and electric field lines (white lines), corresponding to the circular insulation core (a) and electric cloaking (b) conditions. Numerical (c,d) results for the pressure distribution (colormap),  equipressure lines (black lines) and streamlines (white lines), corresponding to pressure-driven flow (c) and hydrodynamic cloaking (d) conditions. Here in the case of electric cloaking, the permittivity is $\varepsilon_s=5/3 \,\varepsilon_m$, and in the case of cloaking the zeta potential is $\tilde{\zeta}_0=-0.16 \,\mathrm{V}$. }\label{fig:circle}
\end{figure}

We next perform finite-element simulations of the flow around an elliptic cylinder in a Hele-Shaw cell, in which the elliptic cylinder consists of an elliptic core coated by a shell.  We consider an elliptic cylinder of boundary $\xi_i$ wrapped by a region of interior and exterior boundaries $\xi_i$ and $\xi_e$, choosing $\xi_i= 0.5$ and $\xi_e= 1$, where the elliptic radii are normalized by the characteristic length $100 \  \mu \mathrm{m}$.  It is remarked that the inner and outer ellipses are of the same focus. We consider a uniform electric field and velocity externally applied in the $\tilde{x}_1$ direction. The zeta potentials $\tilde{\zeta}_0$ of the cloaking region depend on the equation (\ref{ellipse-cloaking-zeta-x}) and the characteristic value of $\tilde{\zeta}_0$. Figures \ref{fig:ellipse}(a) and \ref{fig:ellipse}(b) present the resulting electric field distribution (colormap), equipotential lines (black lines) and electric field lines (white lines).  Figures \ref{fig:ellipse}(c) and \ref{fig:ellipse}(d) present the resulting pressure distribution (colormap) and streamlines (white lines). Analogously, when the background electric field and velocity are externally applied in the $\tilde{x}_2$ direction, under conditions  (\ref{ellipse-cloaking-zeta-y}) we can obtain numerical results illustrated in Figures \ref{fig:ellipse-y}(a)--\ref{fig:ellipse-y}(d).  These results also show excellent agreement like the elliptic cylinder case. The performance of the proposed simultaneous electric and hydrodynamic cloaking conditions has been numerically confirmed.
\begin{figure}[H]
	\centering  %图片全局居中
	\subfigbottomskip=0pt %两行子图之间的行间距
	\subfigcapskip=-10pt %设置子图与子标题之间的距离
	\subfigure[]{
		\includegraphics[width=0.23\linewidth]{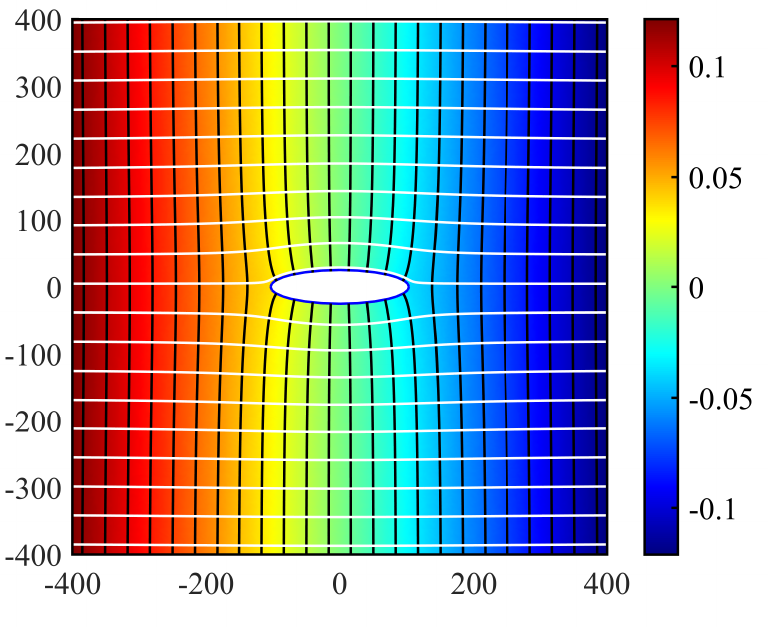}}
	\subfigure[]{
		\includegraphics[width=0.23\linewidth]{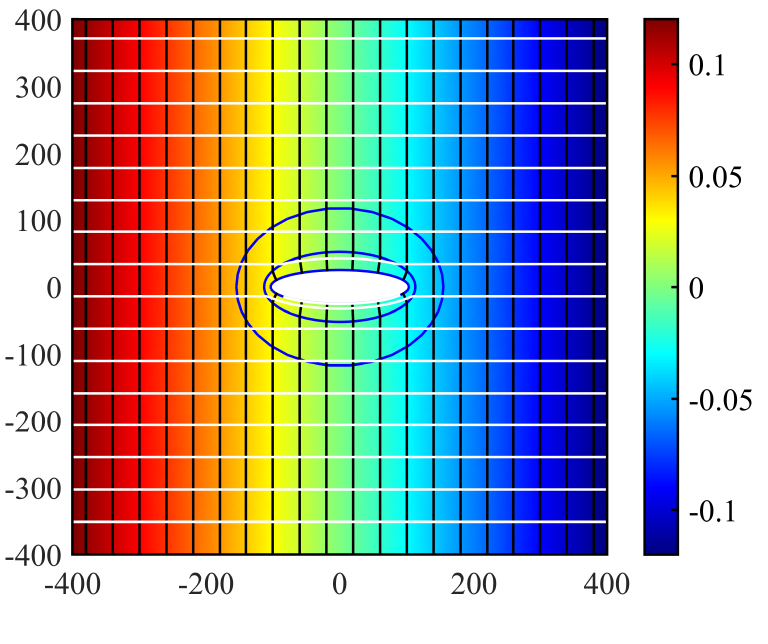}}
	\subfigure[]{
		\includegraphics[width=0.23\linewidth]{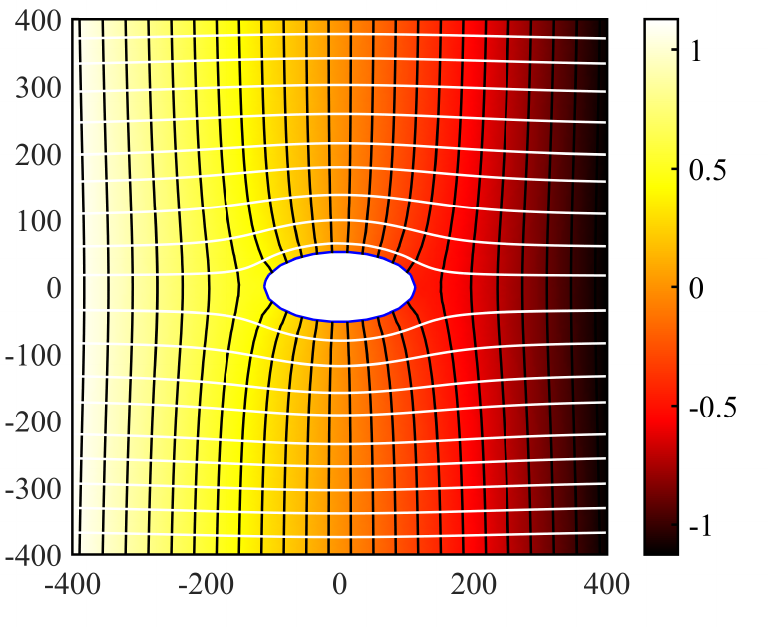}}
	\subfigure[]{
		\includegraphics[width=0.23\linewidth]{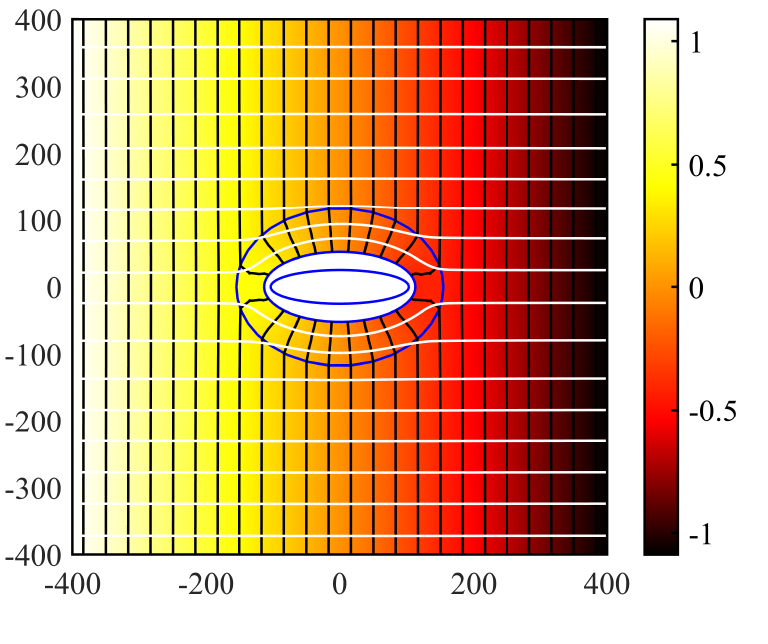}}
%	\subfigure[]{
%		\includegraphics[width=0.32\linewidth]{cloaking-circle.pdf}}
	\caption{Comparison of finite-element simulation results on the confocal ellipses with background field in the $\tilde{x}_1$. Numerical (a,b) results for the electrostatic potential distribution (colormap),  equipotential lines (black lines) and electric field lines (white lines), corresponding to the elliptic insulation core (a) and electric cloaking (b) conditions. Numerical (c,d) results for the pressure distribution (colormap),  equipressure lines (black lines) and streamlines (white lines), corresponding to pressure-driven flow (c) and hydrodynamic cloaking (d) conditions. Here in the case of electric cloaking, the permittivity is $\varepsilon_s=1.887 \,\varepsilon_m$, and in the case of cloaking the zeta potential is $\tilde{\zeta}_0=-0.1555 \,\mathrm{V}$. }\label{fig:ellipse}
\end{figure}

\vspace{-0.5cm}

\begin{figure}[H]
	\centering  %图片全局居中
	\subfigbottomskip=-10pt %两行子图之间的行间距
	\subfigcapskip=-10pt %设置子图与子标题之间的距离
	\subfigure[]{
		\includegraphics[width=0.23\linewidth]{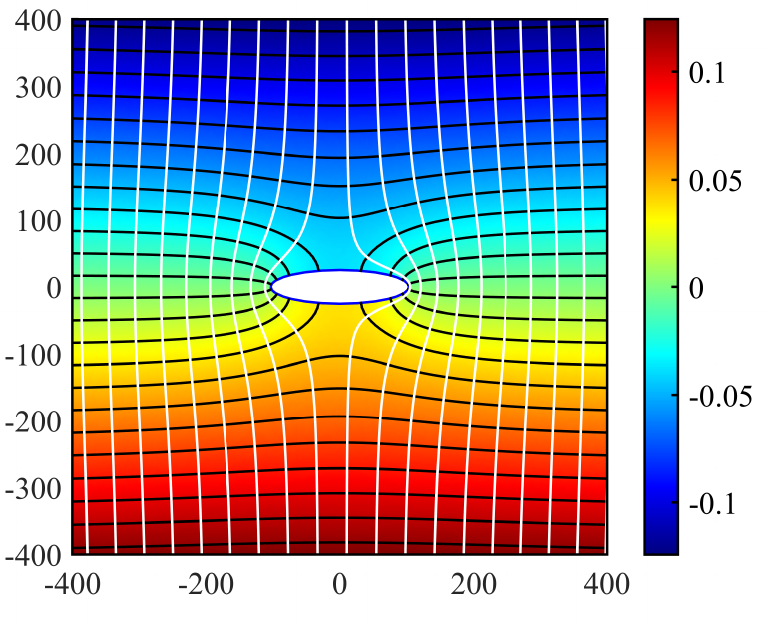}}
	\subfigure[]{
		\includegraphics[width=0.23\linewidth]{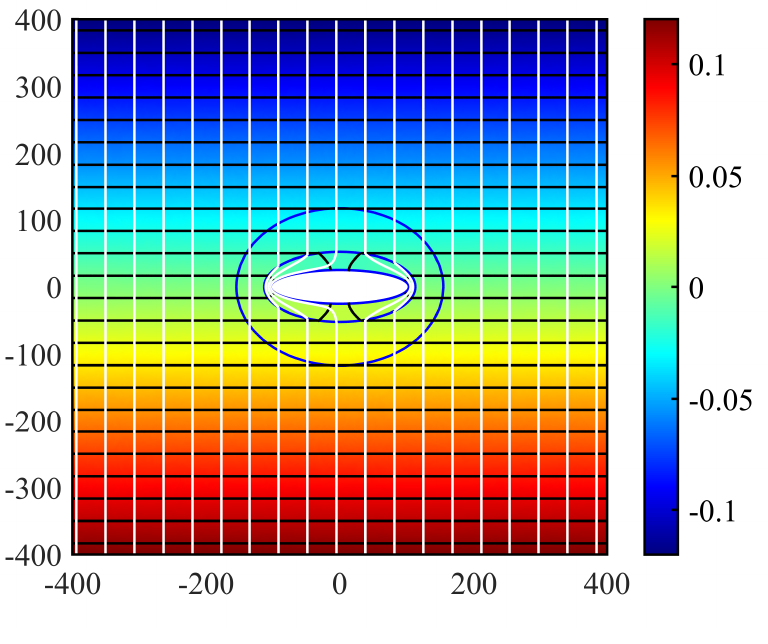}}
	\subfigure[]{
		\includegraphics[width=0.23\linewidth]{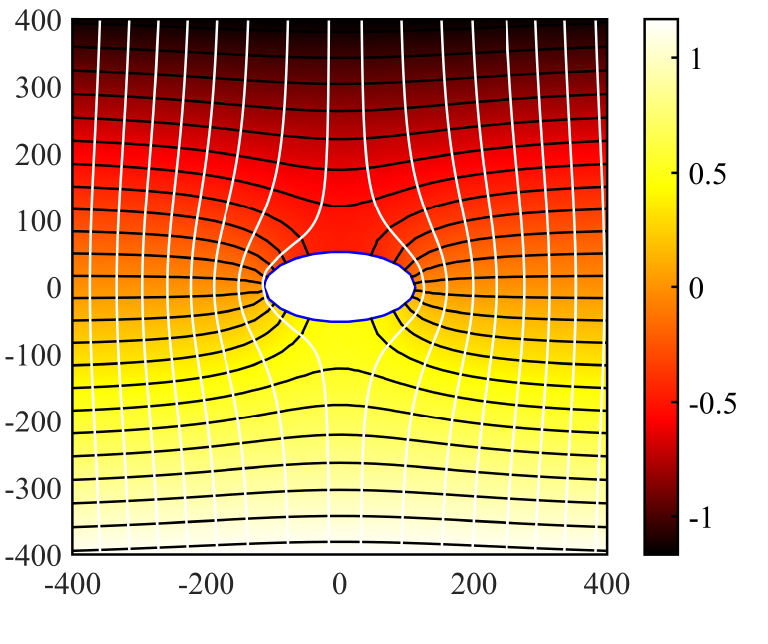}}
	\subfigure[]{
		\includegraphics[width=0.23\linewidth]{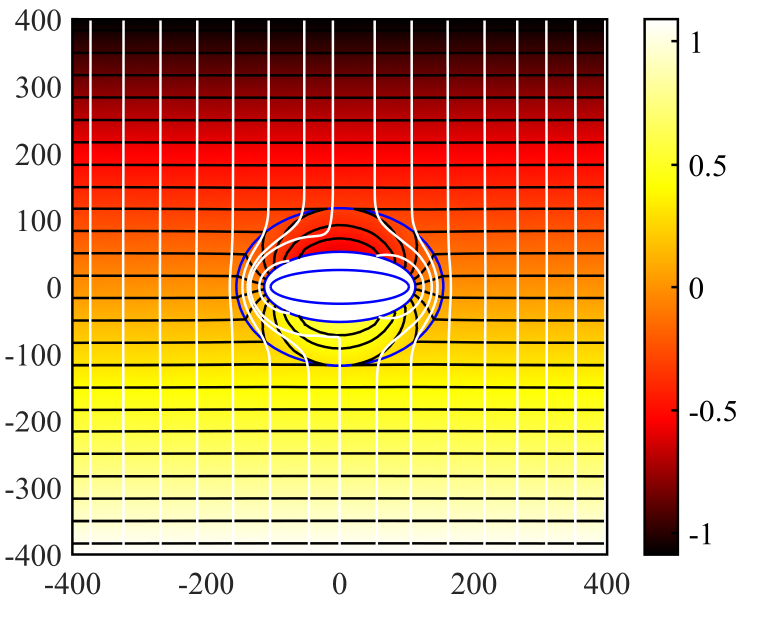}}
%	\subfigure[]{
%		\includegraphics[width=0.32\linewidth]{cloaking-circle.pdf}}
	\caption{Comparison of finite-element simulation results on the confocal ellipses with background field in the $\tilde{x}_2$. Numerical (a,b) results for the electrostatic potential distribution (colormap),  equipotential lines (black lines) and electric field lines (white lines), corresponding to the circular insulation  core (a) and electric cloaking (b) conditions. Numerical (c,d) results for the pressure distribution (colormap),  equipressure lines (black lines) and streamlines (white lines), corresponding to pressure-driven flow (c) and hydrodynamic cloaking (d) conditions. Here in the case of electric cloaking, the permittivity is $\varepsilon_s=8.8354 \,\varepsilon_m$, and in the case of cloaking the zeta potential is $\tilde{\zeta}_0=-0.4419 \,\mathrm{V}$. }\label{fig:ellipse-y}
\end{figure}

In Figure \ref{fig:thin}, we consider thin cloaking region for annulus and confocal ellipses with $r_o =0.9$, $r_i =1$,   $r_e =1.1$, $\xi_o =0.4$, $\xi_i =0.5$ and $\xi_e =0.6$, which are normalized by the characteristic length $100 \  \mu \mathrm{m}$. These results show that thin regions can also have excellent electric and hydrodynamic cloaking. However, these regions are not so thin that meta-surfaces occur because the cloaking conditions are singular.
\begin{figure}[H]
	\centering  %图片全局居中
	\subfigbottomskip=0pt %两行子图之间的行间距
	\subfigcapskip=-10pt %设置子图与子标题之间的距离
	\subfigure[]{
		\includegraphics[width=0.25\linewidth]{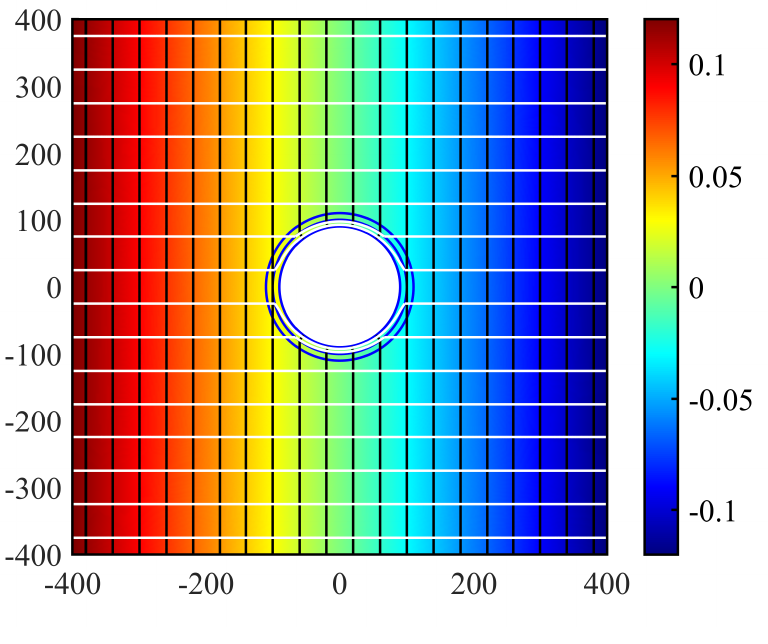}}
	\quad
	\subfigure[]{
		\includegraphics[width=0.25\linewidth]{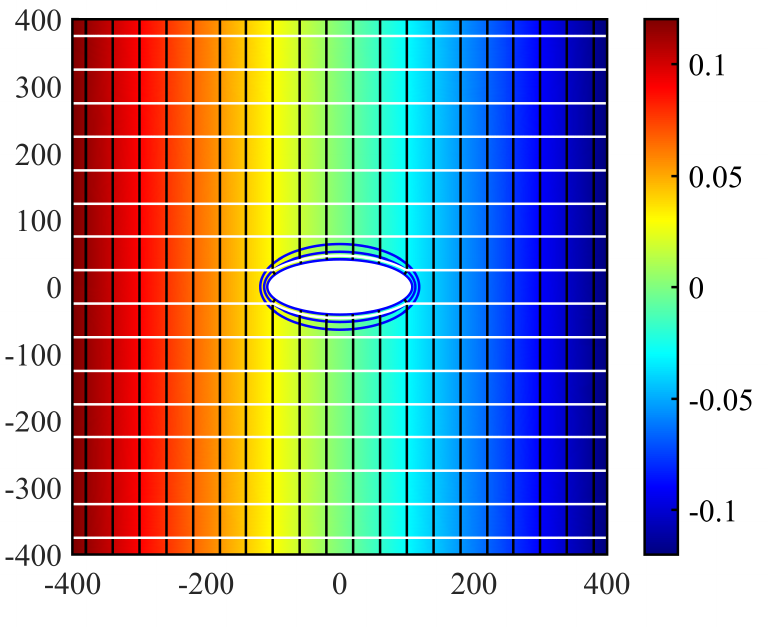}}
	\quad
	\subfigure[]{
		\includegraphics[width=0.25\linewidth]{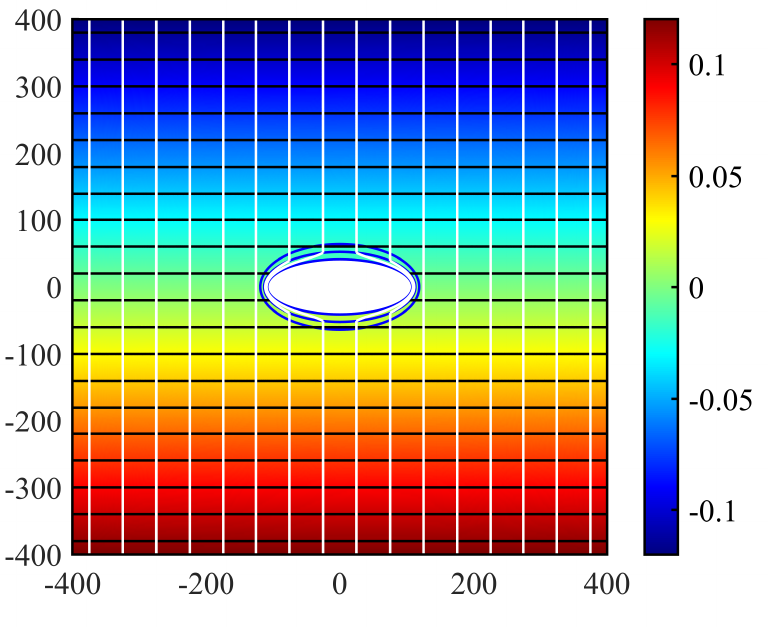}}\\
	\subfigure[]{
		\includegraphics[width=0.25\linewidth]{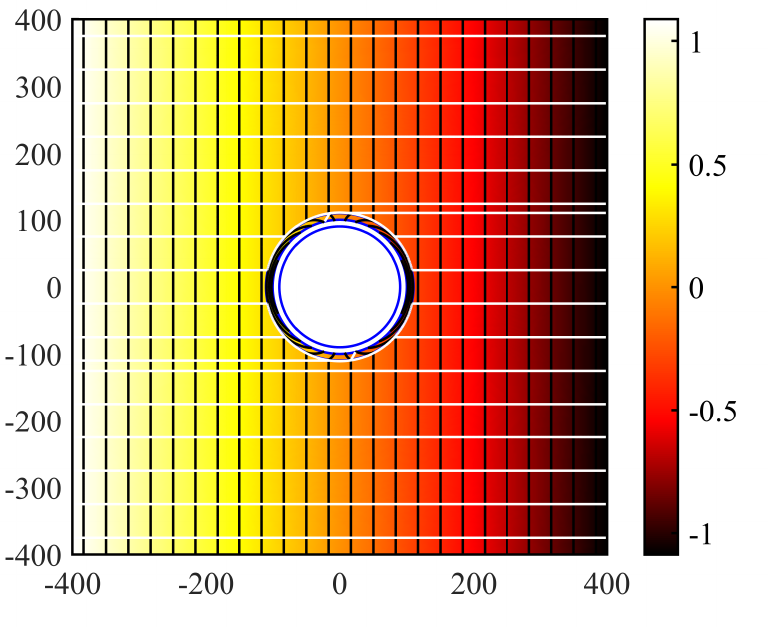}}
	\quad
	\subfigure[]{
		\includegraphics[width=0.25\linewidth]{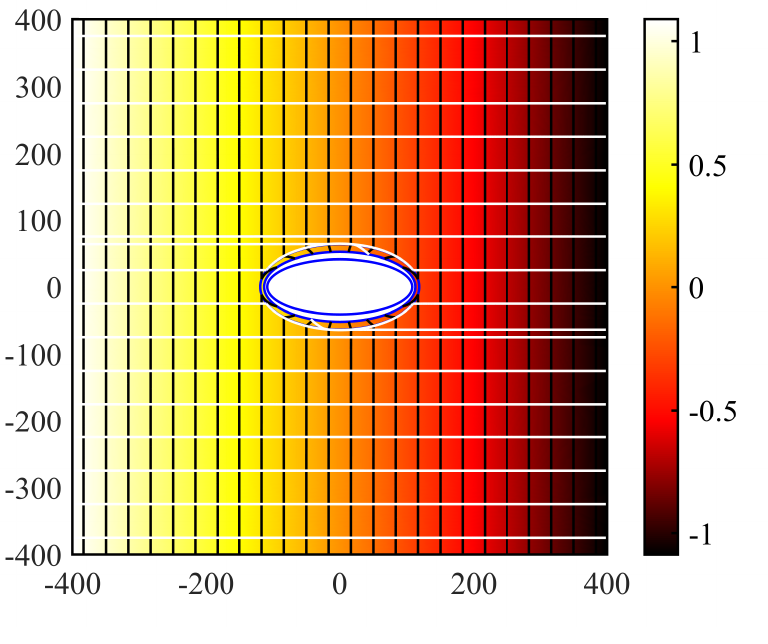}}
	\quad
	\subfigure[]{
		\includegraphics[width=0.25\linewidth]{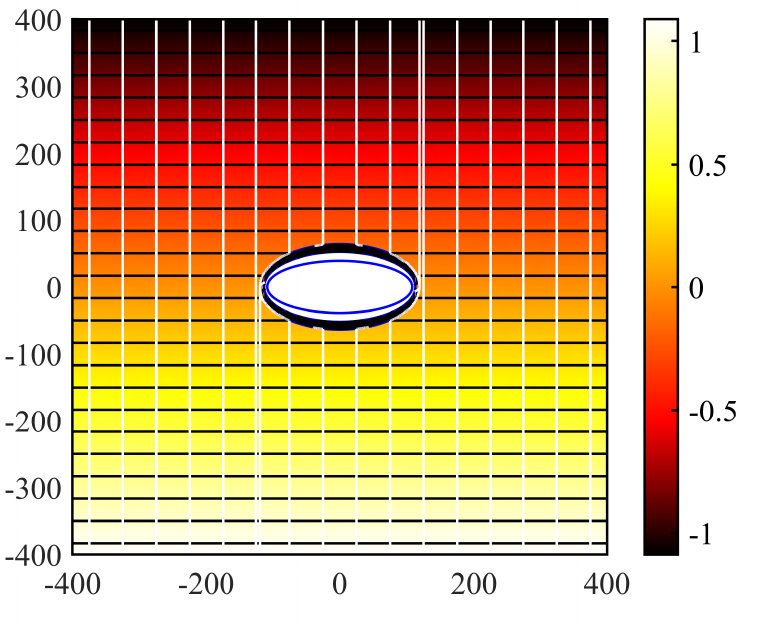}}
	\caption{Comparison of finite-element simulation results with thin cloaking region. The permittivities are $\varepsilon_s=9.5263 \,\varepsilon_m (a), 4.6366 \,\varepsilon_m (b) $, and $21.7116\,\varepsilon_m (c)$. The zeta potentials are $\tilde{\zeta}_0 = -2.2857\ \mathrm{V} (d), -0.8777  \ \mathrm{V} (e)$ and $-4.2438  \ \mathrm{V} (f)$, respectively.}\label{fig:thin}
\end{figure}
\vspace{-0.5cm}
\begin{figure}[htbp]
	\centering  %图片全局居中
	\subfigbottomskip=0pt %两行子图之间的行间距
	\subfigcapskip=-10pt %设置子图与子标题之间的距离
	\subfigure[]{
		\includegraphics[width=0.25\linewidth]{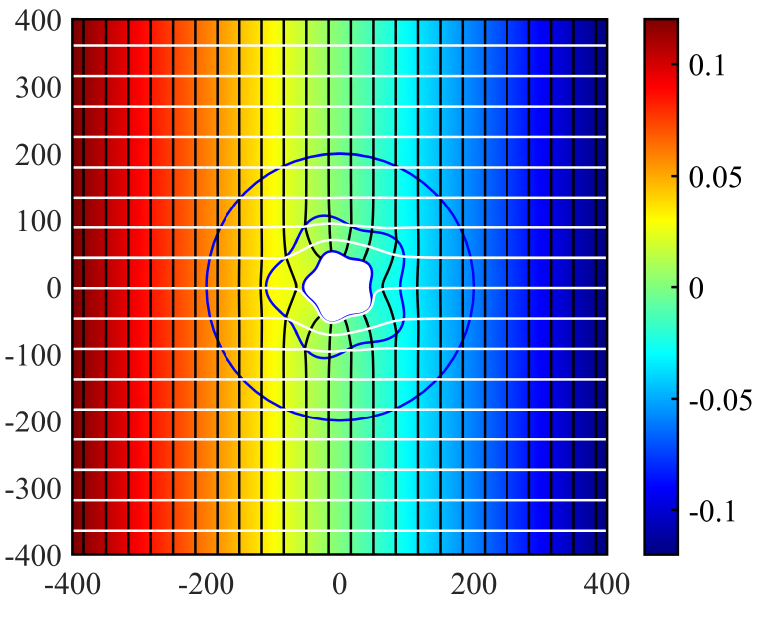}}
	\quad
	\subfigure[]{
		\includegraphics[width=0.25\linewidth]{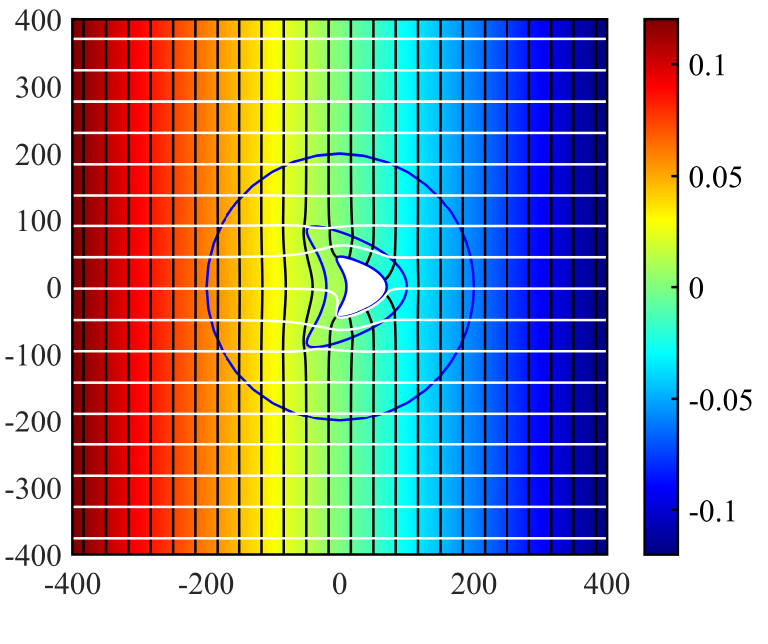}}
	\quad
	\subfigure[]{
		\includegraphics[width=0.25\linewidth]{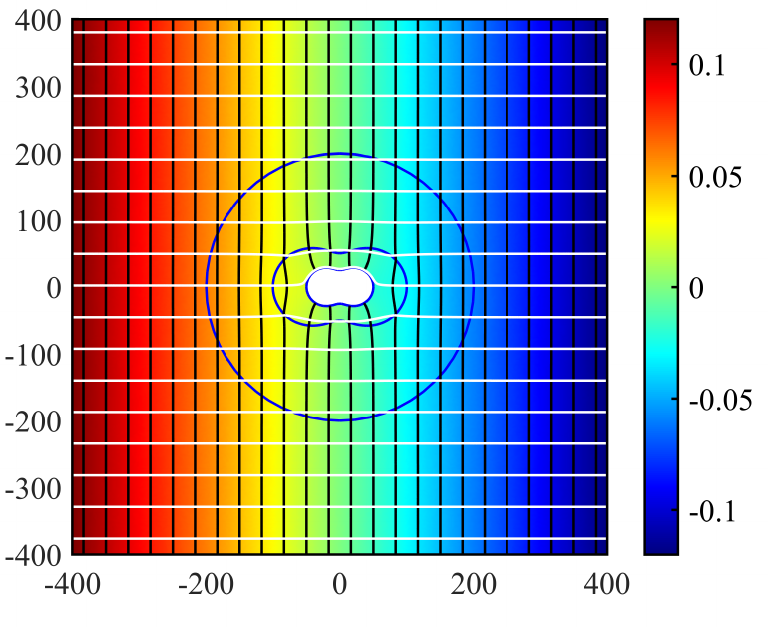}}\\
	\subfigure[]{
		\includegraphics[width=0.25\linewidth]{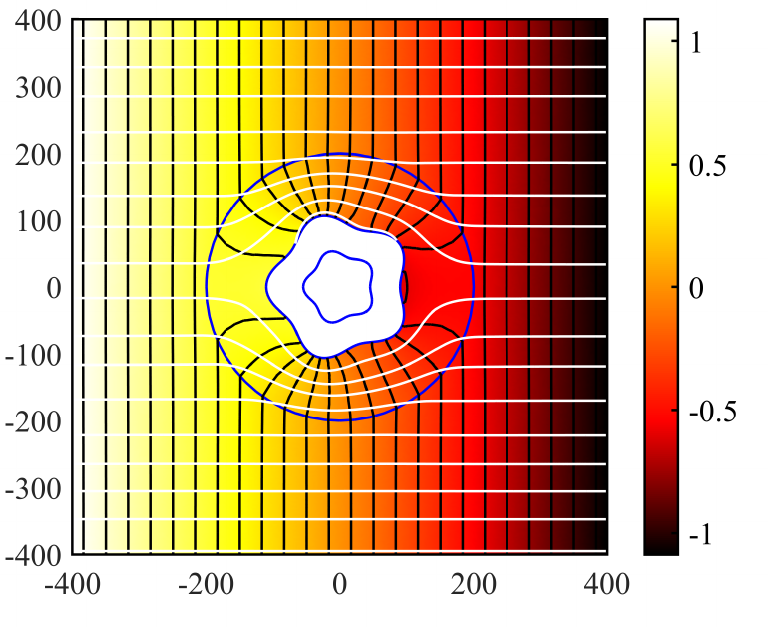}}
	\quad
	\subfigure[]{
		\includegraphics[width=0.25\linewidth]{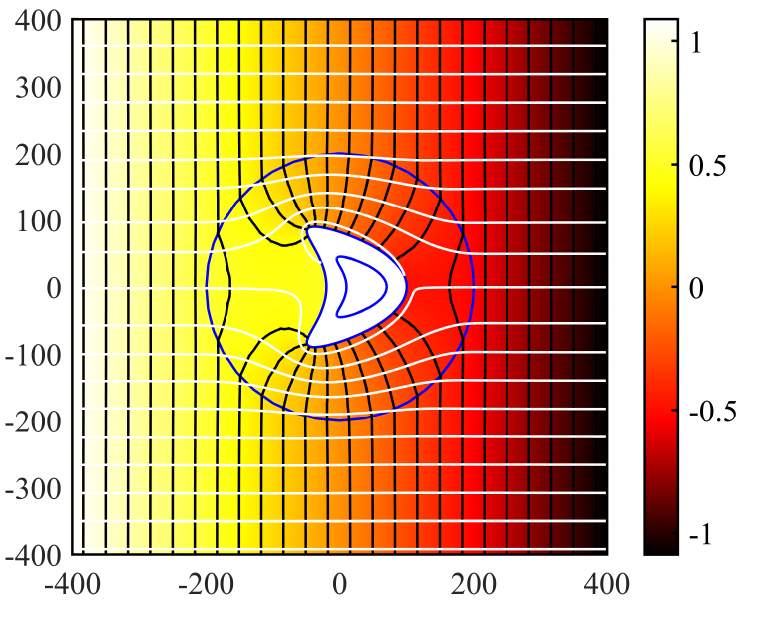}}
	\quad
	\subfigure[]{
		\includegraphics[width=0.25\linewidth]{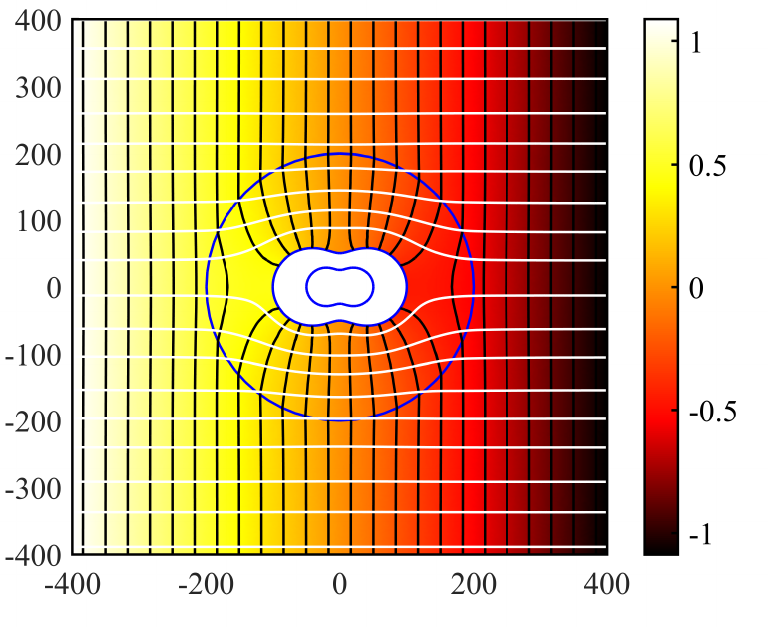}}
	\caption{Comparison of finite-element simulation results for some objects with regular boundaries. Numerical (a-f) results for the electric, pressure distribution (colormap) and electric field lines, streamlines (white lines) corresponding to electric cloaking (a-c) and hydrodynamic cloaking (d-f). The permittivities are $\varepsilon_{s,opt}=1.71 \,\varepsilon_m (a), 1.92 \,\varepsilon_m (b) $, and $1.58 \,\varepsilon_m (c)$. The zeta potentials are $\tilde{\zeta}_{0, opt} = -0.19\ \mathrm{V} (d), -0.1225  \ \mathrm{V} (e)$ and $-0.0809  \ \mathrm{V} (f)$, respectively. The cloaking regions are circles with a radius of $200 \ \mu \mathrm{m}$.}\label{fig:regular-shape}
\end{figure}

For a numerical example, we consider the flow of a cylinder with a flower-shaped cross-section with boundary described by the parametric representation
\begin{equation*}
  x(t) = 1-0.1\cos 5t, \quad 0\leq t \leq 2\pi.
\end{equation*}
Extending our analysis to more complex shapes, we also considered the case of a non-convex kite-shaped object parameterized by
\begin{equation*}
  x(t) = (0.6\cos t + 0.39\cos 2t + 0.01, 0.9\sin t), \quad 0\leq t \leq 2\pi,
\end{equation*}
and a peanut-shaped object parameterized by
\begin{equation*}
  x(t) = \sqrt{\cos^2 t + 0.25 \sin^2 t}, \quad 0\leq t \leq 2\pi.
\end{equation*}
We assume that the boundary curve $\p B$ of the core is conformal to $\p D$, but its size is reduced by half.
When these objects are surrounded by an appropriate circle, we observe that good cloaking occurs in Figure \ref{fig:regular-shape}.
As for the following cases about corners, the assumption on the conformal relation of $\p B$ and $\p D$ still holds.

Without loss of generality, we confine our presentation to objects with boundary curve $\p D$ with some corners.  We consider some special shapes, for instance, triangles, squares and pentagons inscribed in the circle of radius $100 \ \mu \mathrm{m}$.   Here the cloaking region is a circle of radius $200 \ \mu \mathrm{m}$. Figure \ref{fig:corner-shape} shows good cloaking. These permittivities and zeta potentials used in finite-element numerical simulations are calculated by the optimization method in Subsection \ref{subsec:optimal}. They are scaled by the characteristic value of $\tilde{\zeta}_0$.

\begin{figure}[H]
	\centering  %图片全局居中
	\subfigbottomskip=0pt %两行子图之间的行间距
	\subfigcapskip=-10pt %设置子图与子标题之间的距离
	\subfigure[]{
		\includegraphics[width=0.25\linewidth]{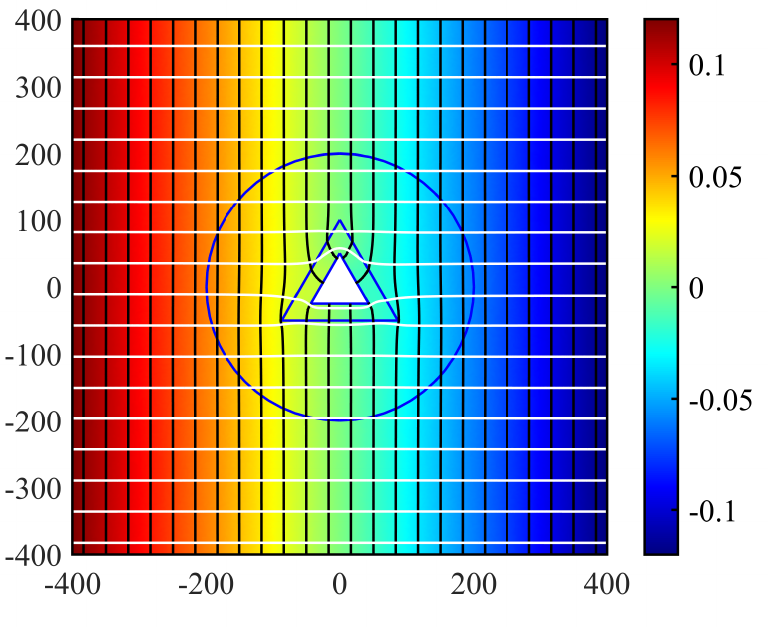}}
	\quad
	\subfigure[]{
		\includegraphics[width=0.25\linewidth]{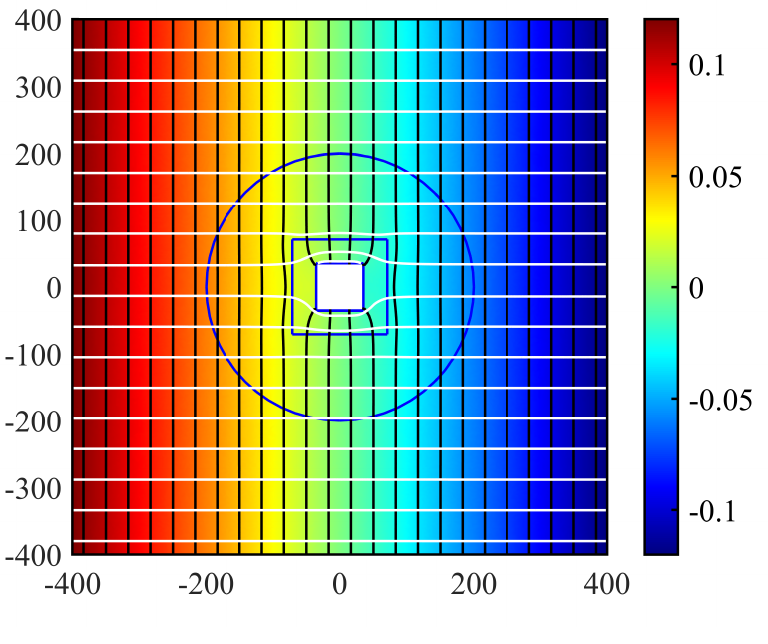}}
	\quad
	\subfigure[]{
		\includegraphics[width=0.25\linewidth]{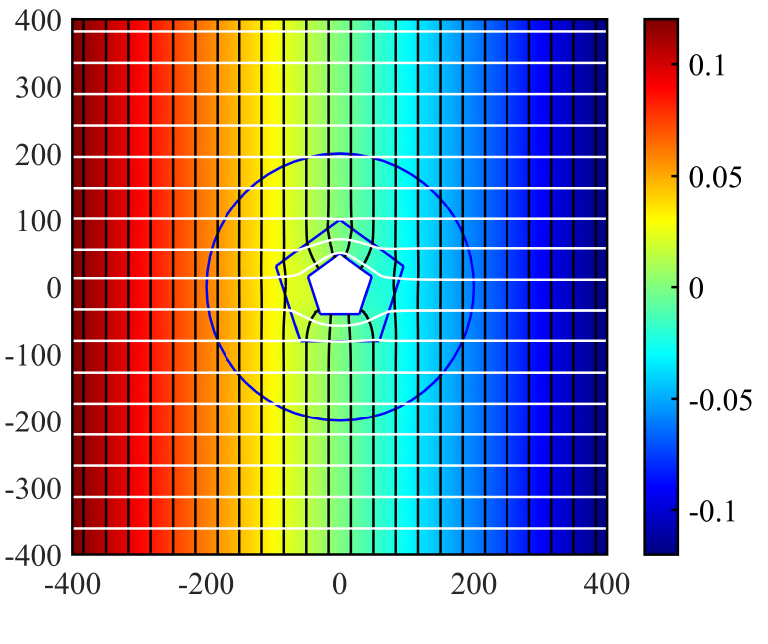}}\\
	\subfigure[]{
		\includegraphics[width=0.25\linewidth]{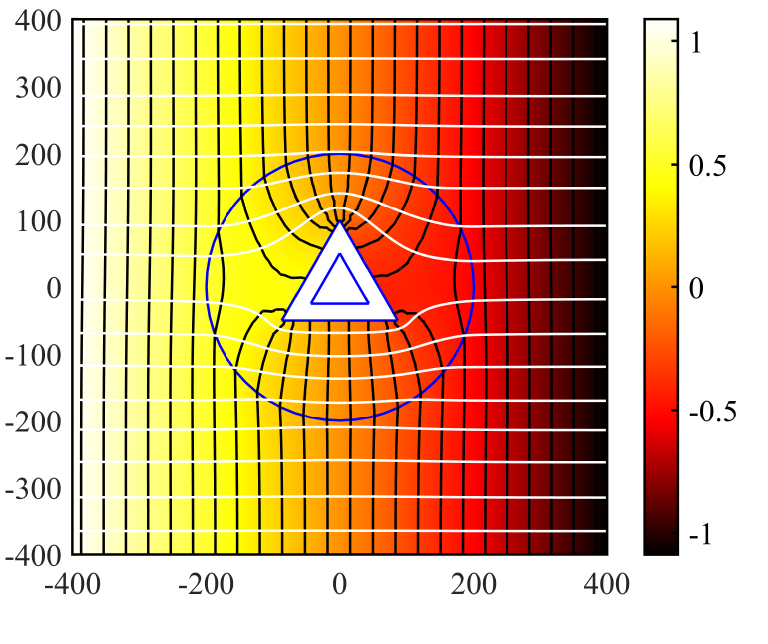}}
	\quad
	\subfigure[]{
		\includegraphics[width=0.25\linewidth]{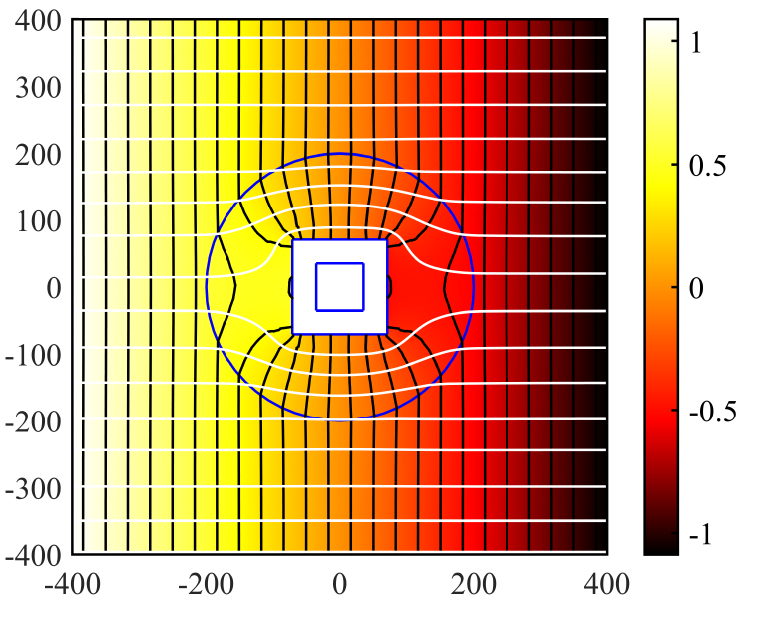}}
	\quad
	\subfigure[]{
		\includegraphics[width=0.25\linewidth]{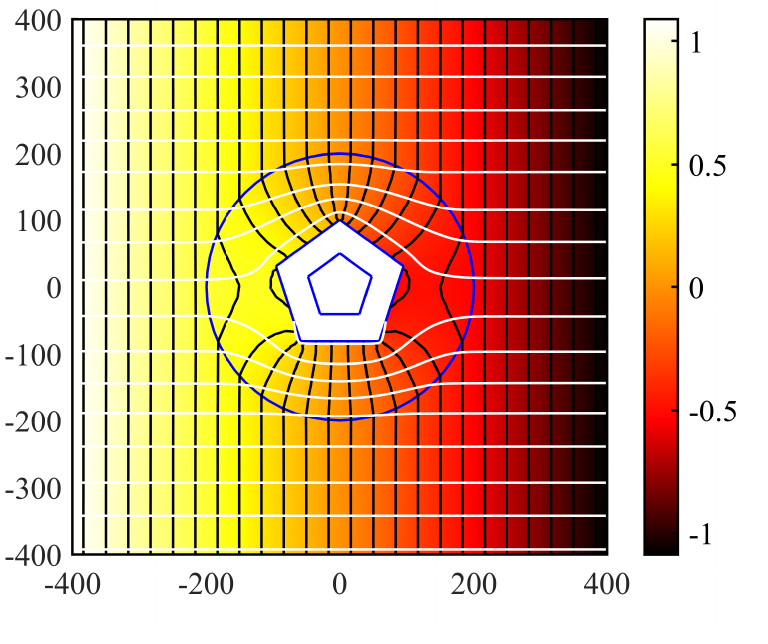}}
	\caption{Comparison of finite-element simulation results for some objects with corners. Numerical (a-f) results for the electric, pressure distribution (colormap) and electric field lines, streamlines (white lines) corresponding to electric cloaking (a-c) and hydrodynamic cloaking (d-f). The permittivities are $\varepsilon_{s,opt}=1.95 \,\varepsilon_m (a), 1.72 \,\varepsilon_m (b) $, and $1.68 \,\varepsilon_m (c)$.  The zeta potentials are $\tilde{\zeta}_{0, opt} = -0.091\ \mathrm{V} (d), -0.119  \ \mathrm{V} (e)$ and $-0.136  \ \mathrm{V} (f)$, respectively. }\label{fig:corner-shape}
\end{figure}

 We extend numerical simulations by investigating the possibility of cloaking multiple objects placed in close proximity to each other. Figure \ref{fig:cloaking-multi-circle-ellipse} shows good cloaking for the combination of circular and elliptical objects. Here excellent cloaking still exists.
%\begin{figure}[H]
%	\centering  %图片全局居中
%	%\subfigbottomskip=2pt %两行子图之间的行间距
%	\subfigcapskip=-10pt %设置子图与子标题之间的距离
%	\subfigure[]{
%		\includegraphics[width=0.32\linewidth]{cloaking-two-circle.pdf}}
%	%\quad
%	\subfigure[]{
%		\includegraphics[width=0.32\linewidth]{cloaking-three-circle.pdf}}
%	\subfigure[]{
%		\includegraphics[width=0.32\linewidth]{cloaking-four-circle.pdf}}
%	\caption{Cloaking for multiple circular cylinder objects. Here $\tilde{r}_i=100 \ \mu \mathrm{m}$, $\tilde{r}_e=200 \ \mu \mathrm{m}$ and $\tilde{\zeta}_0=-0.128 \,\mathrm{V}$.}\label{fig-cloaking-multi-circle}
%\end{figure}

%\begin{figure}[H]
%	\centering  %图片全局居中
%	%\subfigbottomskip=2pt %两行子图之间的行间距
%	\subfigcapskip=-10pt %设置子图与子标题之间的距离
%	\subfigure[]{
%		\includegraphics[width=0.32\linewidth]{cloaking-two-ellipse.pdf}}
%	%\quad
%	\subfigure[]{
%		\includegraphics[width=0.32\linewidth]{cloaking-three-ellipse.pdf}}
%	\subfigure[]{
%		\includegraphics[width=0.32\linewidth]{cloaking-four-ellipse.pdf}}
%	\caption{Cloaking for multiple elliptic cylinder objects. Here $\tilde{\xi}_i=50 \ \mu \mathrm{m}$, $\tilde{\xi}_e=100 \ \mu \mathrm{m}$ and $\tilde{\zeta}_0=-0.1291 \,\mathrm{V}$.}\label{fig-cloaking-multi-ellipse}
%\end{figure}

\begin{figure}[H]
	\centering  %图片全局居中
	\subfigbottomskip=0pt %两行子图之间的行间距
	\subfigcapskip=-10pt %设置子图与子标题之间的距离
	\subfigure[]{
		\includegraphics[width=0.25\linewidth]{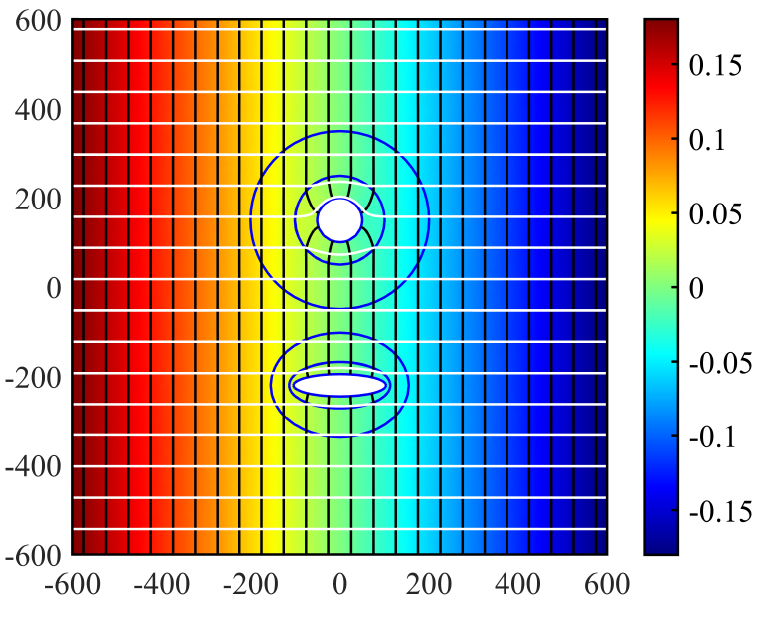}}
	\quad
	\subfigure[]{
		\includegraphics[width=0.25\linewidth]{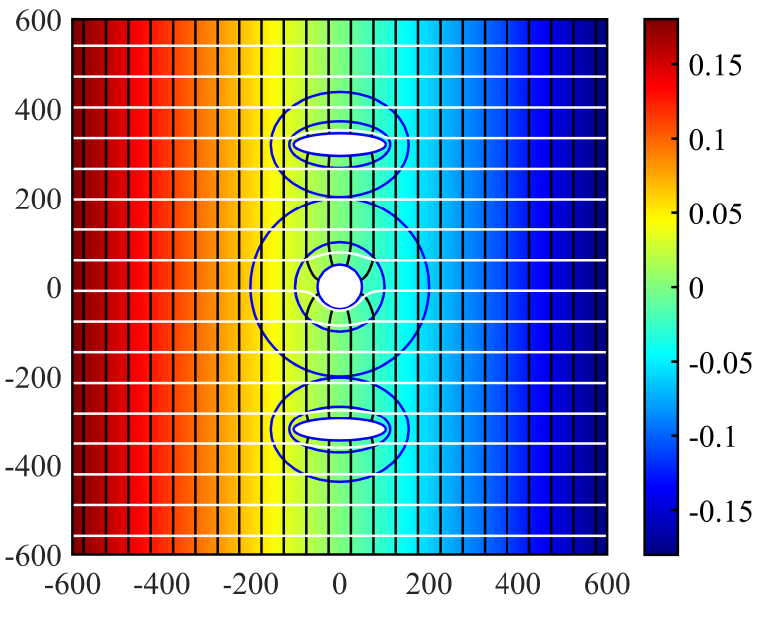}}
	\quad
	\subfigure[]{
		\includegraphics[width=0.25\linewidth]{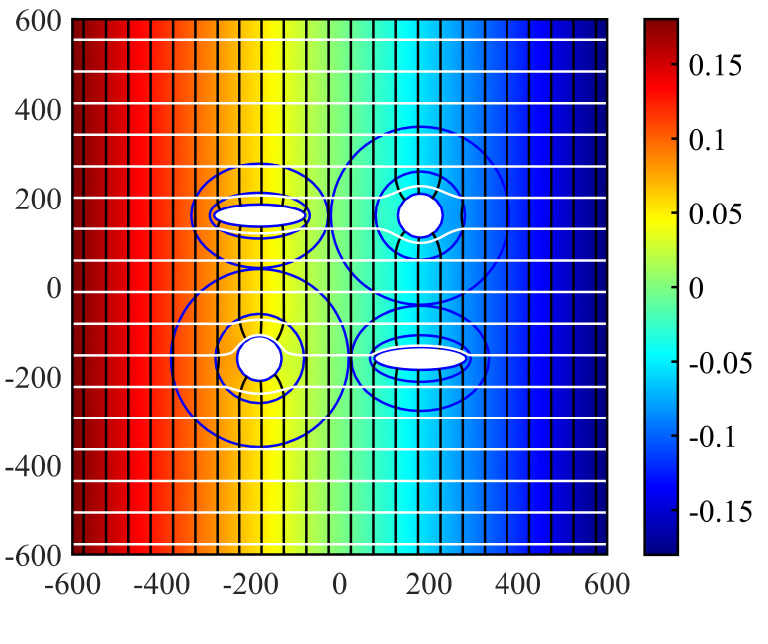}}\\
	\subfigure[]{
		\includegraphics[width=0.25\linewidth]{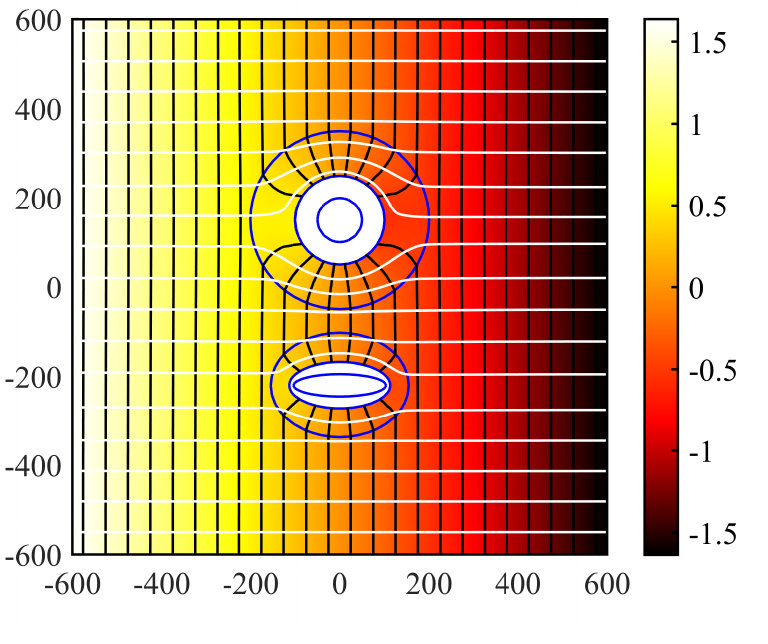}}
	\quad
	\subfigure[]{
		\includegraphics[width=0.25\linewidth]{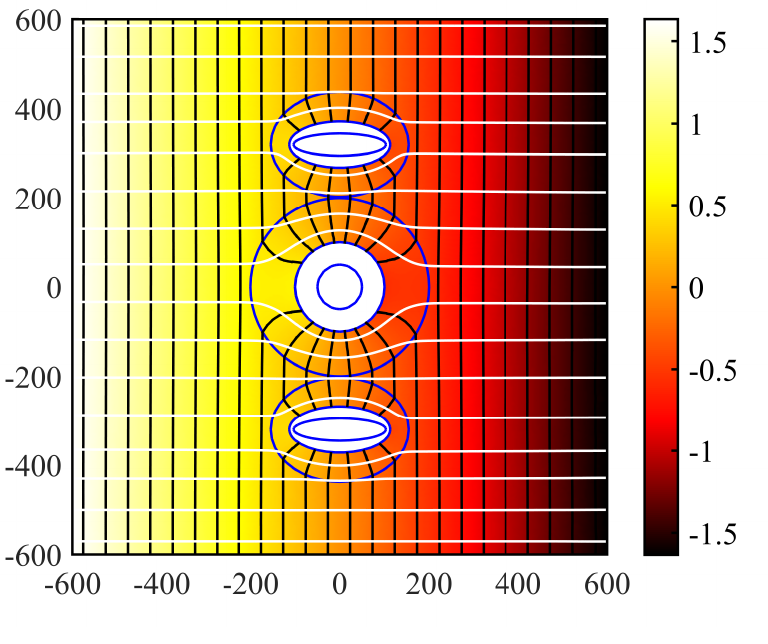}}
	\quad
	\subfigure[]{
		\includegraphics[width=0.25\linewidth]{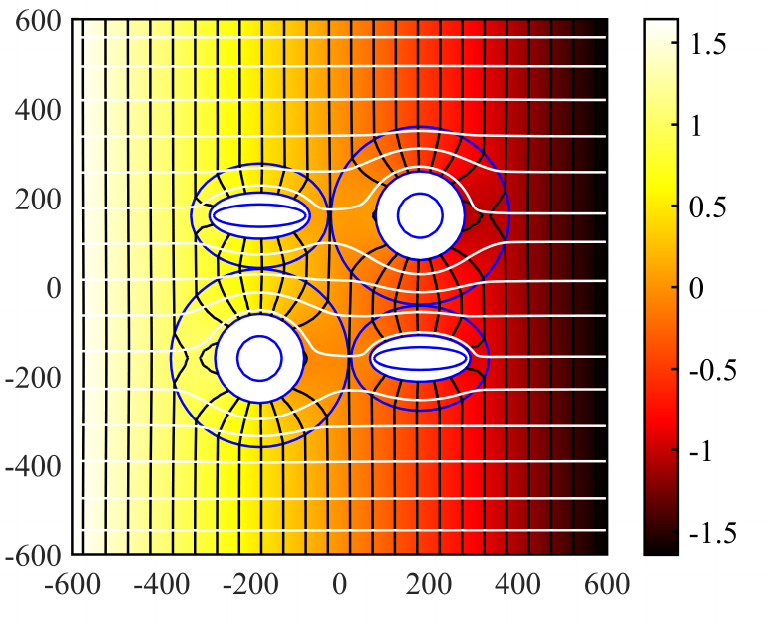}}
	\caption{Simultaneous cloaking for the combination of circular and elliptic cylinder objects. The parameters are the same as that of Figures \ref{fig:circle} and \ref{fig:ellipse}.  }\label{fig:cloaking-multi-circle-ellipse}
\end{figure}

\section{Conclusions}
In this paper, we have established a mathematical framework to design coupled multiphysics cloaking in a Hele-Shaw cell through a combination of scattering-cancellation technology and an electro-osmosis effect. The framework allows us to derive the simultaneous electric and hydrodynamic cloaking conditions for special and  general geometries. In addition to the theoretical results, extensive numerical experiments were conducted to corroborate the theoretical findings. We would like to emphasize that the material parameters needed in our design for each layer are homogeneous, isotropic and metamaterial-less, which could be achieved by natural material. Our study can provide a promising method for feasible and flexible control of multiphysics processes. These results and techniques of this paper can be immediately extended to the illusion case. The new development will be reported in our forthcoming works.

\section*{Acknowledgement}
The research of H. Liu was supported by NSFC/RGC Joint Research Scheme, N CityU101/21, ANR/RGC
Joint Research Scheme, A-CityU203/19, and the Hong Kong RGC General Research Funds (projects 12302919,
12301420 and 11300821). The research of G. Zheng was supported by the NSF of China (12271151), NSF of
Hunan (2020JJ4166) and NSF Innovation Platform Open Fund project of Hunan Province (20K030).
%\\
%\\
%\textbf{\large{Declarations}}
%\\
%\\
%\textbf{Conflict of interest}
%The authors have not disclosed any competing interests.

\end{document}